\documentclass[a4paper,hidelinks]{article}

\usepackage{graphicx}
\usepackage{amsmath}
\usepackage{amssymb}

\usepackage{bm} 
\usepackage[T1]{fontenc}
\usepackage{amsfonts}
\usepackage{mathrsfs}
\usepackage{yfonts}
\usepackage{url}
\usepackage[osf]{mathpazo}  

\usepackage{mathtools}
\usepackage{footmisc}

\usepackage{graphicx}
\usepackage{verbatim}
\usepackage{enumerate}	
\usepackage[english]{babel}
\usepackage{ebproof}		
\usepackage{quoting}
\usepackage{array,booktabs}
\usepackage[all]{xy}\SelectTips{cm}{}
\usepackage{pifont}

\usepackage[shortlabels]{enumitem}
\usepackage{units}
\usepackage{xspace}\xspaceaddexceptions{\,}

\usepackage{booktabs,caption}
\usepackage{tikz-cd}
\quotingsetup{font=small}	

\usepackage{amsthm}

\usepackage{todonotes}

\theoremstyle{definition}   
\newtheorem{definition}{Definition}

\theoremstyle{plain}    
\newtheorem{lemma}[definition]{Lemma}
\newtheorem{proposition}[definition]{Proposition}
\newtheorem{theorem}[definition]{Theorem}
\newtheorem{corollary}[definition]{Corollary}

\renewcommand{\phi}{\varphi}
\newcommand\Le{\mathbf{L}}

\newcommand\pair[1]{{\langle#1\rangle}}
\newcommand\lang{\mathscr{L}}
\newcommand\n{\noindent}

\newcommand\Fm{\mathbf{Fm}}
\newcommand\logic{\mathbf{\Lambda}}
\newcommand\A{\mathbf{A}}
\newcommand\B{\mathbf{B}}
\newcommand\C{\mathbf{C}}
\newcommand\M{\mathfrak{M}}
\newcommand\class[1]{{\mathcal{#1}}}
\newcommand\kripkeframe{\mathfrak{F}}

\newcommand\Se{\mathbf{S}}
\newcommand\gru{\multimap}

\newcommand\Paizero{\mathbf{E}\text{-}\mathbf{PAI}_0^g}
\newcommand\ModelPaizero{\textit{PAI}_0^E}
\newcommand\Paifull{\mathbf{E}\text{-}\mathbf{PAI}^g}
\newcommand\ModelPaifull{\textit{PAI}^E}
\newcommand\Paioriginal{\mathbf{PAI}}
\newcommand\ModelPaioriginal{\textit{PAI}^F}
\newcommand\FineGlobalPai{\mathbf{PAI}^g}
\newcommand\FergusonPai{\mathbf{CA/PAI}}
\newcommand\FergusonPaiBox{\mathbf{CA/PAI}_m}
\newcommand\FergusonGlobalPai{\FergusonPaiBox^g}
\newcommand\FergusonDai{\mathbf{CA/DAI}}
\newcommand\Sfourglobal{\mathbf{S4}^g}
\newcommand\Sfourlocal{\mathbf{S4}}
\newcommand\Fmobject{Fm_{\pair{\neg,\lor,\Box,\to}}}
\newcommand\Daifull{g\mathbf{D}}
\newcommand\ModelDaifull{\textit{gD}}
\newcommand\DunnDai{\mathbf{DAI}}
\newcommand\EpsteinD{\mathbf{D}}
\newcommand\Daizero{g\mathbf{D}_0}
\newcommand\ModelDaizero{\textit{gD}_0}
\newcommand\ModelFergusonPai{\textit{CA/PAI}}
\newcommand\ModelFergusonPaiBox{\textit{CA/PAI}_m}

\newcommand\LocalPaifull{\mathbf{E}\text{-}\mathbf{PAI}^l}

\DeclareSymbolFont{symbolsC}{U}{txsyc}{m}{n}
\DeclareMathSymbol{\strictif}{\mathrel}{symbolsC}{74}

\title{Generalized Epstein semantics for Parry systems}
\author{N. Zamperlin\thanks{I would like to thank Francesco Paoli for his invaluable supervision while developing the ideas that led to this paper. I am also grateful to Thomas Ferguson for the precious insights he gave me while discussing together.}}
\date{\textit{Department of Mathematics and Computer Science,} \\
	\textit{University of Cagliari, Italy}}

\begin{document}

\maketitle

\sloppy

\begin{abstract}
\n In this paper I introduce a generalized version of Richard Epstein's set-assignment semantics (\cite{Epstein90_book}). As a case study, I consider how this framework can be used to characterize William Parry's logic of analytic implication and some of its recent variations proposed by \cite{Ferguson23_subject_matter1}. In generalized Epstein semantics the parallel use of two algebras, one for extensional and the other for intensional values, allows to account for various forms of content sharing between formulae, which motivates the choice to investigate Parry systems. Hilbert-style axiomatizations and completeness proofs will be presented for all the considered calculi, in particular as main result I provide a set-assignment semantics for Parry's logic.
\end{abstract}

\begin{section}{Introduction}\label{sec: intro}

Accounting for subject-matter relations in logic has been a long lasting enterprise that still poses an open problem. Formally, the problem is translated into incorporating content sensitive connectives, above all implication, inside the logical calculus. One of the forefathers of this discussion was William T. Parry, who developed \textit{analytic implication}, that is a content sensitive implication which obeys the \textit{proscriptive principle}: ``No formula with analytic implication as main relation holds universally if it has a free variable occurring in the consequent but not the antecedent" (\cite[p. 151]{Parry68}). In a system where this principle is enforced we have a meaningful semantic contribution provided by the contents of formulae. 

Parry's logic was introduced in \cite{Parry33} and, despite not becoming a main trend in the future branch of relevance logics, it was the starting point for the family of containment logics, in which the notion of a correct implication requires containment of content (in a specified direction) between consequent and antecedent.

Among the plethora of logics that followed this intuition, we focus on a recent proposal, that of Francesco Berto's \textit{topic-sensitive intentional modals} (TSIMs; cf. \cite{Berto2022_topics}), which is situated within the project of a logic of conceivability pursued by Berto and his collaborators. The main goal is to obtain an extremely general framework able to accounting for many intensional and hyperintensional phenomena; such framework has been applied to, e.g., knowability relative to information (\cite{Berto_Hawke21_knowability}), reality-oriented mental simulation (\cite{Berto18_imagination}), hyperintensional belief revision (\cite{Berto2018_belief_revision}). On the technical side, TSIMs are binary, variably strict modal operators of the form $X^\phi\psi$, whose intuitive reading is ``given $\phi$, an agent $X$s that $\psi$", for some propositional attitude $X$. Their semantics is provided by a Kripke frame containing a different accessibility relation for each TSIM, expanded furthermore with a semilattice of topics for each world and a map assigning to each formula its content; the resulting semantic clause for $X^\phi\psi$ is that of a strict implication $\phi\strictif\psi$ (through the accessibility relation $R_X$) which moreover satisfies a content filter, in the sense that the content of $\phi$ contains that of $\psi$.

Despite its great versatility, Berto's theory has a crucial decisive question left open: what is the content of a TSIM? A fully transparent solution which reduces the content of $X^\phi\psi$ to the fusion of the contents of the components formulae is unsatisfactory, but until a decision is taken the language of the logic can only be take as first degree with respect to the new intensional operators, limiting its expressive power. An attempt to solve the problem comes for Thomas M. Ferguson, whose study of subject-matter and intensional operators has been developed in a recent series of papers (\cite{Ferguson23_subject_matter1}, \cite{Ferguson23_subject_matter2}, \cite{Ferguson23_subject_matter3}). The first article in the series is the one of major interest for the current paper. In that, Ferguson considers as a starting point of his investigaiton the case study provided by Parry's logic of analytic implication, focusing on the possible world semantics devised by \cite{Fine86_analytic_implication}, which is based on $\mathbf{S4}$-Kripke frames expanded by a semilattice of topics for each world and certain homomorphisms that guarantees the persistence of content inclusion. By assigning to an implicative formula a binary operation over said semillatice and removing any property from this new operation, Ferguson manages to obtain a very weak, therefore vastly general, content sensitive conditional, which can be seen as a special TSIM with a non-trivial topic itself.

While looking for a solution to Berto's framework limitation, this paper takes into consideration the helpful insights provided by \cite{Ferguson23_subject_matter1} but it follows a different route, taking inspiration from the work of Richard L. Epstein and his work on intensional implication developed within the framework of set-assignment semantics. This semantics has been developed by Epstein since \cite{Epstein79}, culminating with the monograph \cite{Epstein90_book}. In set-assignment semantics, the content of formulae is a certain subset of a reference set. By imposing conditions on this assignment and requiring that a formula satisfies certain relation in the content of the components as additional condition for his truth semantic clause, the main connective of said formula is given an intensional layer. This is the general strategy used to intensionalize Boolean connectives without changing their underlying semantics dramatically. 

Set-assignment semantics has been used to successfully recapture well-know logics (see \cite{Epstein90_book}): modal (ch. VI; e.g., $\mathbf{K},\mathbf{T},\mathbf{TB},\mathbf{S4},\mathbf{S5}$, Grzegorczyk's $\mathbf{S4Grz}$, provability logic $\mathbf{GL}$), intuitionistic (ch. VII; including Johansson's minimal logic), many-valued (ch. VIII; $\mathbf{K_3}$, \L ukasiewicz's $\mathbf{\L_3}, \mathbf{\L_\infty}$, G\"odel's $\mathbf{G_3}, \mathbf{G_\infty}$), paraconsistent (ch. IX; D'Ottaviano and da Costa's $\mathbf{J_3}$). Of a much greater interest for our current purpose is the original family of dependence logics introduced by Epstein (ch. V). These logics will be properly introduced in the next section.

\end{section}

\begin{section}{Systems for content inclusion}\label{sec: parry_intro}

Parry's logic of analytic implication $\Paioriginal$ is presented by \cite{Fine86_analytic_implication}\footnote{As Ferguson notes, Fine's logic is a system for an extension of Parry's original one. Nonetheless, the expanded system can be considered faithful to Parry's intent (cf. \cite[p. 10, footnote 2]{Ferguson23_subject_matter1}.). See \cite[pp. 4-5]{Ferguson17_book} for further historical details.} for a language\footnote{In the following we will mainly work with the language $\lang_\mathbf{ML}^\to=\pair{\neg,\lor,\Box,\to}$, defining $\prec$ and, in the stronger systems, defining $\Box$ as well. In particular, in Fine's logic we can define $\phi\prec\psi:=\psi\to(\phi\lor\neg\phi)$, as he acknowledges in the quoted paper. For this reason in the following we are going to take $\prec$ as defined, in this way we will be able to compare ours and Fine's logics as they belong to the same similarity type.\label{footnote: Fine_language}} of type $\pair{\neg,\land,\to\,\Box,\prec}$ as the following Hilbert-style system\footnote{\cite{Fine86_analytic_implication} actually provides two systems, the second of which is Parry's logic as presented by \cite{Dunn72_demodalized_implication}. Fine proves that the two systems are equivalent.}:

\begin{itemize}
    \item[(1)] $\phi$, for any classical tautology $\phi$
    \item[(2)] $((\phi\prec\psi)\land(\psi\prec\chi))\supset(\phi\prec\chi)$
    \item[(3)] $((\phi\prec\chi)\land(\psi\prec\chi))\supset((\phi\land\psi)\prec\chi)$
    \item[(4)] $\phi\prec\psi$, when $Var(\phi)\subseteq Var(\psi)$
    \item[(5)] $(\phi\prec\psi)\supset\Box(\phi\prec\psi)$
    \item[(6)] $(\phi\to\psi)\equiv(\Box(\phi\supset\psi)\land(\psi\prec\phi))$
    \item[(K)] $\Box(\phi\supset\psi)\supset(\Box\phi\supset\Box\psi)$
    \item[(T)] $\Box\phi\supset\phi$
    \item[(4)] $\Box\phi\supset\Box\Box\phi$
    \item[(MP)] $\phi,\phi\to\psi\vdash\psi$
    \item[(Nec)] $\vdash\phi\text{ }\Rightarrow\text{ }\vdash\Box\phi$
\end{itemize}

Fine's model theory for $\Paioriginal$ is given by an augmented Kripke-style semantics, in which to each world $w$ is associated a set of contents or topics $T_w$, with a mapping $t_w$ assigning content to formulae relative to a world. The algebraic structure of $T_w$ is actually that of a semilattice, with the induced partial order intuitively read as content inclusion.

\begin{definition}\label{def.: Fine_PAI_model}
    A $\ModelPaioriginal$-model is a tuple $\pair{W,R,\pair{T_w,\oplus_w}_{w\in W},v,\{t_w\}_{w\in W}}$. $\pair{W,R}$ is a $\mathbf{S4}$-frame. For each $w\in W$, $\pair{T_w,\oplus_w}$ is a join-semilattice, and $t_w:Fm\to T_w$ is a mapping satisfying $t_w(\neg\phi)=t_w(\phi)$ and $t_w(\phi\circ\psi)=t_w(\phi)\oplus t_w(\psi)$ for binary $\circ$. Moreover for $p,q\in Var$, if $w'\in R[w]$ then $t_w(p)\leq_w t_w(q)$ implies $t_w(p)\leq_{w'} t_{w'}(q)$. $v:Var\to\mathscr{P}(W)$ is a valuation.
\end{definition}

\n Persistence of content inclusion extends to all formulae, that is $t_w(\phi)\leq_w t_w(\psi)$ implies $t_w(\phi)\leq_{w'} t_{w'}(\psi)$ when $w'\in R[w]$. The forcing relation in a model $\M$ is standard, with the addition of:
\begin{itemize}
    \item $\M,w\Vdash\phi\prec\psi$ iff $t_w(\phi)\leq_w t_w(\psi)$;
    \item $\M,w\Vdash\phi\to\psi$ iff $t_w(\phi)\leq_w t_w(\psi)$ and $\forall w'\in R[w], w'\Vdash\phi\supset\psi$.
\end{itemize}

\n The notions of satisfiability and validity are standard. In this setting, by frame we mean $\pair{W,R,\pair{T_w,\oplus_w}_{w\in W}}$. $\Paioriginal$ is proved to be complete w.r.t. the class of $\ModelPaioriginal$-frames (\cite{Fine86_analytic_implication}).

In the following we will discuss global modal logics, here we recall the basic definition.

\begin{definition}
    Let $\class{F}$ be a class of Kripke frames. The \textit{global} logical consequence relation induced by $\class{F}$ is defined as: $\Gamma\vDash^g_\class{F}\phi\Leftrightarrow$ for all models $\M=\pair{\kripkeframe,v}$ with $\kripkeframe\in\class{F}$, $\M\vDash\gamma$ for each $\gamma\in\Gamma$ implies $\M\vDash\phi$.
\end{definition}

\n The usual notion of modal consequence relation will be called the \textit{local} relation and will be denoted by $\vDash^l$. We say that a logic is complete in the local (respectively, global) sense w.r.t. a class of frames $\class{F}$ when the consequence relation over $\class{F}$ is intended as the local (global) one.

In the context of minimal normal modal logic \textbf{K} and its extensions, the global companion of a logic is obtained by switching the rule of necessitation with the rule of global necessitation.

\begin{definition}
    Let $\logic$ be an expansion of \textbf{K} whose Hilbert-style presentation contains the rule (Nec). Its \textit{global companion} $\logic^g$ is the logic obtained by substituting (Nec$_g$) $\phi\vdash\Box\phi$ for (Nec).
\end{definition}
    
\n It is immediate to see that $\logic\subseteq\logic^g$.

Considering an extension of \textbf{K} which is complete in the local sense w.r.t. a class of frames, its global companion is complete in the global sense w.r.t. to the same class of frames. We sketch the proof for $\mathbf{K}^g$, it can be easily adapted to its axiomatic extensions obtained adding formulae describing frame properties.

\begin{proposition}\label{prop.: global_K_completeness}
    $\mathbf{K}^g$ is complete w.r.t. the global consequence relation over the class of all Kripke frames.
\end{proposition}
\begin{proof}[Sketch of the proof]
    Take any Henkin-style completeness proof of \textbf{K} (e.g., \cite[ch. 4.2]{Blue_book}). Consider a \textbf{K}$^g$ theory $\Gamma$. The canonical model $\M^\Gamma=\pair{W,R,v}$ for $\Gamma$ differs from a standard canonical model only for the fact that we put $W=\{ \Delta\supseteq\Gamma$ $|$ $\Delta$ is a maximal \textbf{K}-consistent theory$\}$. Observe that we require each $\Delta$ to be only a theory of local \textbf{K}. It can be proved in the standard way\footnote{A little care is needed when proving the existence lemma (\cite[p. 199, lem. 4.20]{Blue_book}), which states that in the canonical model for each point $w$, if $\diamond\phi\in w$ then $\exists w'\in R[w], \phi\in w'$. In the local case this is proved by building the set $v^-=\{\phi\}\cup\{\psi$ $|$ $\Box\psi\in w\}$, prove it to be consistent, expand it via Lindenbaum's lemma and show that said expansion is accessible from $w$. In the global case, we also need to prove that $\Gamma\subseteq v^-$, otherwise the expansion of $v^-$ might not belong to $W$, since our canonical model contains only maximal consistent sets that expand $\Gamma$. This is guaranteed by the fact that $\Gamma$ is a \textbf{K}$^g$-theory, hence it is closed under (Nec$_g$), that is $\forall\gamma\in\Gamma,\Box\gamma\in\Gamma$. Since $w$ is a point of the canonical model, $\Gamma\subseteq w$, so $\Gamma=\{\gamma$ $|$ $\Box\gamma\in\Gamma\}\subseteq\{\psi$ $|$ $\Box\psi\in w\}\subseteq v^-$.} that the canonical valuation $v$ is indeed a valuation function. Now proceed by contraposition and suppose $\Gamma\nvdash_{\mathbf{K}^g}\phi$. Expand $\Gamma$ to the least $\mathbf{K}^g$-theory $\Gamma'\supseteq\Gamma$. Again $\Gamma'\nvdash_{\mathbf{K}^g}\phi$, which implies $\Gamma'\nvdash_{\mathbf{K}}\phi$. By Lindenbaum's lemma, expand $\Gamma'$ to a maximal \textbf{K}-consistent theory $\Delta$ s.t. $\phi\notin\Delta$. Consider the canonical model for $\Gamma'$. We have $\Delta\in W$ and for all $\Sigma\in W,\Sigma\Vdash\gamma$ for each $\gamma\in\Gamma'$, since $\Gamma'\subseteq\Sigma$. At the same time $\Delta\nVdash\phi$. Putting $\class{K}$ the class of all frames, we conclude $\Gamma\nvDash^g_\class{K}\phi$.
\end{proof}

Since $\Paioriginal$ is an expansion of $\Sfourlocal$, we can consider its global companion $\FineGlobalPai$.

\begin{proposition}\label{th.: pai_global_kripke_completeness}
    $\FineGlobalPai$ is complete w.r.t. the global consequence relation over the class of \textit{PAI}-frames.
\end{proposition}
\begin{proof}
    Consider \cite{Fine86_analytic_implication}'s proof\footnote{To be precise, \cite{Fine86_analytic_implication} only proves weak completeness for $\Paioriginal$. Nonetheless the proof of strong completeness can be easily obtained adapting his. It is enough to have a working Lindenbaum's lemma, which is provided in the following by theorem \ref{th.: lindenbaum_like_lemma}.} and his canonical model. Starting from a $\Paioriginal^g$-theory $\Gamma$ we perform the change operated in the proof of proposition \ref{prop.: global_K_completeness}, i.e. defining $W$ as the set of all maximal $\Paioriginal$-consistent theories expanding $\Gamma$. Then the proof runs as in proposition \ref{prop.: global_K_completeness}.
\end{proof}

To obtain a system weak enough to be closer to Berto's project for a very general semantics, \cite{Ferguson23_subject_matter1} removes any interaction between the content of formulae with analytic implication and those formed using the other connectives. This is a refinement of Fine's suggestion in \cite{Fine86_analytic_implication} that analytic  implication is itself a concept and therefore it provides a transformative contribution to the content of any implicative formula. By detaching the action of fusion from the operation associated with the intensional arrow, Ferguson obtains a semantics in which there is no presupposition about the constraints over the topic of a conditional. 

\begin{definition}\label{def.: ferguson_ca_pai_model}
    A $\ModelFergusonPai$-model is a tuple $\pair{W,R,\langle T_w,\oplus_w,\gru_w}_{w\in W},v,\{t_w\}_{w\in W},\{h_{ww'}\}_{w,w'\in W}\rangle$, which differs from a \textit{PAI}-model only for the following: for all $w\in W, \gru_w:T_w^2\to T_w$, $t_w(\phi\to\psi)=t_w(\phi)\gru_w t_w(\psi)$, and if $w'\in R[w], h_{w,w'}:T_w\to T_{w'}$ is a homomorphism s.t. $h_{w,w'}(t_w(p))=t_{w'}(p)$, for $p\in Var$.
\end{definition}

\n The use of homomorphisms $h_{w,w'}$ is a tool to obtain persistence content preservation across the successors of a world, in fact it holds: if $w'\in R[w]$, $t_w(\phi)\leq_w t_w(\psi)$ implies $t_w(\phi)\leq_{w'} t_{w'}(\psi)$.

This class characterizes the logic of \textit{conditional-agnostic analytic implication} $\FergusonPai$ (\cite{Ferguson23_subject_matter1}), a sublogic of $\Paioriginal$ in the language $\pair{\neg,\land,\lor,\to}$. Observe that this semantics is truly agnostic for what it concerns implication, $\gru$ is simply a groupoid operation for which no non-trivial property holds: e.g., for $x,y,z\in T_w$, neither $x\gru_w y=x\oplus_w y$ nor $x\gru_w y=y\gru_w x$, nor it is compatible with the order of $\oplus_w$, e.g. $x\leq_{\oplus_w} y$ doesn't imply $z\gru_w x\leq_{\oplus_w}z\gru_w y$.

A logic very close to Parry's one, Dunn's logic of demodalized analytic implication \textbf{DAI} (\cite{Dunn72_demodalized_implication}) is considered in \cite{Ferguson23_subject_matter1} as a related case study. In the standard setting, Dunn's logic is obtained by adding to the axiomatic base $\Paioriginal$ the demodalizer axiom $\phi\to(\neg\phi\to\phi)$ (the arrow here is read as strict implication), which collapses any modal logic as strong as \textbf{T} into classical logic. The advantage of such move is to obtain a non-modal propositional logic which still enjoys the variable sharing properties of Parry's logic. Ferguson performs on $\Daifull$ the same changes done in the passage from $\Paioriginal$ to $\FergusonPai$: Dunn's logic is weakened to $\FergusonDai=\FergusonPai+(\phi\supset((\phi\supset\phi)\to\phi))$\footnote{The use of material implication in the axiom is motivated by the desire for generality. Otherwise constraints over topic inclusion should be added to the class of models that constitutes the semantics of $\FergusonDai$ to guarantee soundness.}. In this way he obtains a non-modal logic in which the conditional has no constraint over its contribution to the content of a formula. A model for this logic can either be defined as a model for $\FergusonPai$ in which every world is blind, or a single point Kripke frame, that is a classical model. The reference to demodalized logics is helpful in relation to the last author we now present.

A different approach towards a fully compositional account for intensional connectives stems from the framework of Richard Epstein's \textit{set-assignment semantics}. In \cite{Epstein90_book}, he provides a general framework for various propositional logics, furthermore he introduces two original families of intensional logical systems: \textit{dependence} and \textit{relatedness}\footnote{The interesting case of relatedness logics \textbf{S} and \textbf{R} will not be considered in this paper. The basic idea motivating the symmetric subject-matter relatedness logic \textbf{S} is that two formulae are content related when their contents overlap. The same holds for non-symmetric subject-matter relatedness logic \textbf{R}, but the picture here is more fine-grained and formulae are assigned a left and a right content, and we require the left content of the antecedent of an implication to overlap with the right content of the consequent for the two formulae to be related. See \cite{Epstein90_book} for an extensive discussion on relatedness logics. The first set-assignment semantics for \textbf{R} was provided by \cite{Krajewski91}.} logics. In the following, we are going to focus on dependence logics.

The basic intuition behind Epstein's programme (\cite{Krajewski91}) is that the logically significant properties of a sentence are not only its truth-value but its content as well. Moreover, the truth of a conditional sentence is guaranteed both by its classically material value and by a certain relation between the contents of antecedent and consequent. To extend this approach, the strategy to add an intensional layer to a Boolean connective without deviating too much from classical logic, is to add a content constraint between the components occurring under its scope. Epstein focuses on implication\footnote{\cite[p. 121]{Epstein90_book} briefly considers the case of two dependent disjunctions as well.}. In order to talk about contents, two approaches are presented: relation-based\footnote{The relation-based approach is going through a renaissance as part of the expanded project of relating semantics championed by the Torunian logic group. The main reference here is the detailed introduction by \cite{Klonowski_relating_logic_history_2021}; see also \cite{Jarmuzek_Paoli_2021}. Still in the Polish school but from a different perspective, a recent investigation of Epstein's work is \cite{Krawczyk_math_2024}.} and function-based. Set-assignment semantics belongs to latter approach, where a formula is directly assigned a content via a function that maps it to a subset of a certain fixed set.

\begin{definition}
    Let $Fm$ be the formulae in the language $\pair{\neg,\land,\lor,\to}$; given a countable set $U$, a \textit{set-assignment} function is a mapping $s:Fm\to\mathscr{P}(U)$. It is called a \textit{union set-assignment} when $s(\phi)=\bigcup\{s(p)$ $|$ $p\in Var(\phi)\}$.
\end{definition}

\n Therefore union set-assignments obey a principle of content compositionality.

Dependence logics are still two-valued, although implication becomes an intensional connective, therefore its value is also determined by the content relationship between it components. Dependence logics rise when this relationship is intended as content inclusion.

\begin{definition}\label{def.: dependence_model}
    A \textit{dependence model} is a pair $\pair{v,s}$ where $s$ is a union set-assignment and $v:Fm\to\{1,0\}$ is a Boolean valuation w.r.t. $\neg,\land,\lor$, while:
    \[ v(\phi\to\psi) = \begin{cases} 
	v(\neg\phi\lor\psi) & \text{if } s(\phi)\supseteq s(\psi); \\
	0 & \text{otherwise.}
    \end{cases}
    \]
\end{definition}

The class of dependence models is axiomatized by the logic $\EpsteinD$. This is actually a reformulation of Dunn's logic of demodalized analytic implication $\DunnDai$. Changing the relation on the first condition of the clause for $\to$ we obtain semantics characterizing different logics:

\begin{itemize}
    \item Dual dependence logic \textbf{dD}: $s(\phi) \subseteq s(\psi)$;
	
    \item Logic of equality of content \textbf{Eq}: $s(\phi) = s(\psi)$.
\end{itemize}

\cite{Epstein90_book} provides complete Hilbert-style calculi for all these logics. Ultimately, the essence of dependence logics is to allow the usual Boolean structure underlying classical logic with a set of contents which behaves like a powerset. If we look more carefully though, the entire powerset structure is redundant, what is actually playing a role is the possibility to fuse topics together, therefore we are really using only the join-semilattice reduct of $\mathscr{P}(U)$. This is the first step towards a generalization of Epstein's semantics.


	
	
	



\end{section}

\begin{section}{Framework}\label{sec: framework}

In the following, if \textbf{S} is an algebra, we denote its support by \textit{S}. Moreover, given a similarity type $\tau$, $\Fm_\tau(X)$ is the absolutely free algebra of type $\tau$ generated by a countable set $X$. We fix $X$ and use it for generating all the formula algebras mentioned in the following, therefore we will simply write $\Fm_\tau$, omitting the set of generators.

Returning to Epstein's construction of dependence models, we can unravel his procedure:
\begin{itemize}
    \item[(1)] fix the 2-element Boolean algebra as the algebra of truth-values for $Fm_{\pair{\neg,\land,\lor,\to}}$ (formulae in the Boolean language expanded with $\to$);
    \item[(2)] translate $Fm_{\pair{\neg,\land,\lor,\to}}$ into the set-theoretic language of type $\pair{\cup}$ according to the union set-assignment prescription;
    \item[(3)] fix a powerset and evaluate the translated formulae $Fm_{\pair{\cup}}$ via the mapping $s$;
    \item[(4)] using $s$, assign truth-values to the whole $Fm_{\pair{\neg,\land,\lor,\to}}$ with $v$.
\end{itemize}

The expanded Boolean language $\pair{\neg,\land,\lor,\to}$ is evaluated thanks to an auxiliary function which assign contents to formulae. Only via a detour through $\mathscr{P}(U)$ we can recursively apply $v$ to the entire object language. This additional content comes into play only for the fragment of language which exceeds the Boolean one, that is for $\to$, this is why valuations are still homomorphisms w.r.t. standard Boolean formulae.

In this strategy there are two limitations: first, there are implicit elements, mainly the algebra of truth-values (the 2-element Boolean algebra) and the way in which the language of this algebra and the powerset communicate; second, the strength of the assumed set-theoretic background, that is using a powerset as the algebra of contents. Allowing a degree of freedom over these contexts leads us to a more customizable procedure. Let $\rho_0\supseteq\rho_A$ be our object language:

\begin{itemize}
    \item[(1)] take an algebra \textbf{A} (of extensional values) of type $\rho_A$;
    \item[(2)] translate $Fm_{\rho_0}$ into the language $\rho_B$;
    \item[(3)] take an algebra \textbf{B} (of intensional values) of type $\rho_B$ and define a valuation $s:Fm_{\rho_B}\to B$;
    \item[(4)] using $s$, define a valuation of $v:Fm_{\rho_0}\to A$.
\end{itemize}

We have finally reached a full generalization\footnote{It is interesting to note how \cite{Nowak_2008}'s investigation of semantics for some axiomatic extensions of Suszko's sentential calculus with identity seems to underlie a similar conception to the one detailed above, with the use of an auxiliary algebra of contents and an implicit notion of translation between types.} of Epstein's semantics:

\begin{definition} \label{def.: gen_Epstein_model}
    A \textit{general Epstein model} ($gE$-model) for a language of type $\rho_0$ is a tuple $\pair{\A,\B, N,v_A,v_B}$. $\A$ is an algebra of type $\rho_A\subseteq\rho_0$, $\B$ is an algebra of type $\rho_B$. $ N: Fm_{\rho_0} \to Fm_{\rho_B}$ is a mapping that satisfies $ N (x) = x$ and $ N(\alpha(x_1/\beta_1, ...$ $x_n/\beta_n)) =  N(\alpha) (x_1/ N(\beta_1), ...$ $x_n/ N(\beta_n))$ for all variables $x, x_1, ...$ $x_n \in Var$, and formulae $\alpha, \beta_1, ...$ $\beta_n \in Fm_{\rho_0}$. Finally, $v_B: Fm_{\rho_B} \to B$ is a homomorphism, while the mapping $v_A: Fm_{\rho_0} \to A$ is required to be a homomorphism only w.r.t. the symbols of $\rho_A$\footnote{It would be natural to introduce the notion of $gE$-frame as a tuple $\pair{\A,\B, N}$, adapting the following definitions of validity to the case of validity in a frame, as validity in every model built over a frame, therefore working on the higher level of validity. We will not pursue that path, because it could lead to confusion with the traditionally consolidated notion of Kripke frame, which will appear frequently as possible world semantics is the main framework of comparison for our general Epstein semantics. Moreover there would be no essential use of the category of frames over that of models.}.
\end{definition}

\n By $\rho_0$ we denote the proper object language, which sometimes we will call \textit{external language}, which in all the interesting cases will exceed the \textit{internal language} $\rho_A$. The intuitive reading of the above structure is that $\A$ can be considered the \textit{algebra of truth-values} and $v_A$ is an assignment of values (though not necessarily a full homomorphism) of the formulae of type $\rho_0$. Observe that the lack of any constraint over the behaviour of $v$ in the case of symbols of $\rho_0$ belonging outside of $\rho_A$ is the mechanism that we will employ to add some specific intensional connotation to the interpretation of said symbol, a connotation that the algebra of truth-values alone is not capable to provide. For this reason the valuation of the symbols strictly belonging to the external language will be determined also by $v_B$. This leads us to the structure $\B$, the \textit{algebra of contents}, and the assignment $v_B$ for formulae of type $\rho_B$. In order to provide a content to the formulae of the object language, namely $\rho_0$, a translation is needed. Said translation\footnote{The use of a translation was inspired by the naming functor introduced by \cite{Gratzer_agassiz_sum_1974} and \cite{Graczynska_agassiz_sum_1978} in their theory of Agassiz sums, a generalization of P\l onka sums. I am grateful to Francesco Paoli for pointing me towards these references.} is the mapping $ N$, which transforms the formulae of $Fm_{\rho_0}$ into those of $Fm_{\rho_B}$ while obeying an appropriate closure under substitutions\footnote{The use of a translation function distinct from the content assigning function $v_B$ follows the intention of making explicit all the elements at work in original Epstein's semantics. In the context of union set-assignment, said translation was already incorporated in the mapping $s$. The same choice can be made here, removing $N$ from the model and adding further properties to $v_B$, although losing the quality of being a homomorphism. There is no essential difference between the two approaches.}. We will sometimes employ the names algebra of extensional values and algebra of intensional values to, respectively, $\A$ and $\B$.

We already introduce a convention that bends slightly definition \ref{def.: gen_Epstein_model}. Despite we require the type of the object language $\rho_0$ to contain the type of the language of the algebra of truth-values $\rho_A$, we will ignore the constant operations from $\rho_A$ as long as they can be expressed as algebraic constants. E.g., we are going to work with Boolean algebras, yet is none of the object languages we will consider there will be the constants $0$ and $1$. This convention will allow us to greatly simplify the axiomatic systems.

Standard dependence models are a particular case of $gE$-models, in fact they are tuples $\pair{\mathbf{B}_2,\pair{\mathscr{P}(U),\cup},N,v,s}$, where \textbf{B}$_2$ is the 2-element Boolean algebra, $N:Fm_\pair{\neg,\land,\lor,\to}\to Fm_\pair{\cup}$ deletes negations and transforms every binary connectives into set-union, $s:Fm_\pair{\cup}\to\mathscr{P}(U)$ is a homomorphism and $v:Fm_\pair{\neg,\land,\lor,\to}\to\{1,0\}$ is as in definition \ref{def.: dependence_model}, with:
\[ 
    v(\phi\to\psi) = \begin{cases} 
    v(\neg\phi\lor\psi) & \text{if } s(N(\phi))\supseteq s(N(\psi)); \\
    0 & \text{otherwise.}
    \end{cases}
\]

Let us call the above a $D^+$-model. If we change the relation in the first condition for $\to$ to $\subseteq$ (respectively, =) we obtain $dD^+$-models ($Eq^+$-models). The following in an easy exercise:

\begin{theorem}
    For $\Le\in\lbrace\mathbf{D},\mathbf{dD},\mathbf{Eq}\rbrace$, $\Le$ is complete w.r.t. the class of $L^+$-models.
\end{theorem}

As we mentioned, there is no essential use of the structure of the powerset $\mathscr{P}(U)$ in the semantics for dependence logic, just the properties of a join-semilattice. Similarly, for what it concerns us the 2-element Boolean algebra has no privileged role over the entire variety of Boolean algebras. Define therefore a $gL$-model, for $\Le\in\lbrace\mathbf{D},\mathbf{dD},\mathbf{Eq}\rbrace$, as a tuple $\pair{\mathbf{B},\pair{S,\oplus},N,v,s}$, s.t. \textbf{B} is a Boolean algebra, $\pair{S,\oplus}$ is a join-semilattice, $N:Fm_\pair{\neg,\land,\lor,\to}\to Fm_\pair{\oplus}$ deletes negations and transforms binary connectives into join, $s:Fm_\pair{\oplus}\to S$ is a homomorphism and $v$ is a Boolean homomorphism w.r.t. Boolean connectives, and:
\[ v(\phi \to \psi) = \begin{cases} 
	v(\neg\phi \lor \psi) & \text{if } C_L; \\
	0 & \text{otherwise.}
\end{cases}
\] 

\n where condition $C_L$ changes according to $\Le$:
\begin{itemize}
	\item[(\textbf{D})] $s(N(\phi))\supseteq s(N(\psi))$;
	\item[(\textbf{dD})] $s(N(\phi))\subseteq s(N(\psi))$;
	\item[(\textbf{Eq})] $s(N(\phi))=s(N(\psi))$.
\end{itemize}

\n The following will be proved as theorem \ref{th.: epstein_logics_correspondence}: for $\Le\in\lbrace\mathbf{D},\mathbf{dD},\mathbf{Eq}\rbrace$, $\Le$ is complete w.r.t. the class of $gL$-models\footnote{I will briefly mention an equivalent theorem for the symmetric relatedness logic \textbf{S} (\cite[ch. III]{Epstein90_book}). In a symmetric relatedness model $\pair{v,s}$ the union set-assignment $s$ is required to be non-empty (i.e., $s(p)\neq\emptyset$, for all $p\in Var$), and the condition for $\phi\to\psi$ to be evaluated as a material implication is that $s(\phi)\cap s(\psi)\neq\emptyset$. This class characterizes the logic \textbf{S}. In generalized Epstein semantics we obtain a complete semantics by employing models of the form $\pair{\mathbf{B},\pair{L,\sqcap,\sqcup,0},N,v,s}$, which differ from the previous $gL$-models for using a lower bounded distributive lattice $\pair{L,\sqcap,\sqcup,0}$ as algebra of contents and the condition $C_L$ becomes ``if $s(N(\phi))\sqcap s(N(\psi))\neq 0"$. The situation for the logic \textbf{R} is more complicated and yet to be studied, as its set-assignment semantics by \cite{Krajewski91} in terms of pairs of assignment functions is less straightforward.}.






\end{section}

\begin{section}{Global logics of analytic implication}\label{sec: global_Parry_logic}

Let $\lang_\mathbf{ML}^\to=\pair{\neg,\lor,\Box,\to}$ be the language of standard modal logic expanded with a binary connective $\to$, which denotes the intensional arrow of analytic implication. The following abbreviations are common: $\phi\land\psi := \neg(\neg\phi\lor\neg\psi)$, $\phi\supset\psi := \neg\phi\lor\psi$, $\phi\equiv\psi := (\phi\supset\psi)\land(\psi\supset\phi)$. Moreover we define:
$$\phi\prec\psi:=\psi\to(\phi\lor\neg\phi)$$

\n In a moment, after introducing the semantics, it will be possible to appreciate the usefulness of this abbreviation.

In the following we are going to employ interior algebras. An \textit{interior algebra} $\A=\pair{A,\land,\lor,\neg,\Box,0,1}$ is a Boolean algebra expanded with a unary $\Box$ operation such that:
\begin{itemize}
    \item[(EqK1)] $\Box 1\approx 1$;
    \item[(EqK2)] $\Box(x\land y)\approx\Box x\land\Box y$;
    \item[(EqT)] $\Box x\approx\Box x\land x$;
    \item[(Eq4)] $\Box x\approx\Box x\land\Box\Box x$.
\end{itemize} 

\n To make proofs easier to read, we are going to actually employ the $\land$-free reduct of interior algebras, calling them interior algebras for simplicity. We keep using the $\land$ operation, treating it as term-defined.

\begin{definition}\label{def.: pai0-model}
    A $\ModelPaizero$-model is a $gE$-model $\pair{\B,\Se, N,v,s}$ for $\lang_\mathbf{ML}^\to$. $\B$ is an interior algebra. $\Se=\pair{S,\oplus,\gru}$ is composed of a join-semilattice $\pair{S,\oplus}$, with $\leq_\oplus$ the induced partial order, extended by a binary operation $\gru:S^2\to S$ without any constraint. The translation $ N:\Fmobject\to Fm_{\pair{\oplus,\gru}}$ satisfies:
    \begin{itemize}
        \item $ N(\neg\phi)= N(\Box\phi)= N(\phi)$;
        \item $ N(\phi\lor\psi)= N(\phi)\oplus N(\psi)$;
        \item $ N(\phi\to\psi)= N(\phi)\gru N(\psi)$.
    \end{itemize}

    \n $s:Fm_{\pair{\oplus,\gru}}\to S$ is a homomorphism. $v:\Fmobject\to B$ is a homomorphism w.r.t. $\neg,\lor,\Box$, moreover it satisfies the following:
    $$v(\phi\to\psi)=\begin{cases} 
	\Box^{\B} v(\neg\phi\lor\psi) & \text{if } s( N(\psi))\leq_\oplus s( N(\phi));\\
        0^{\B} & \text{otherwise.}
        \end{cases}$$
        
\end{definition}


Notice that the abbreviation $\phi\prec\psi$ is actually a syntactic way to express a relation between contents: in any $\ModelPaizero$-model, $v(\phi\prec\psi)=v(\psi\to(\phi\lor\neg\phi))=1$ iff $s( N(\phi))\leq_\oplus s( N(\psi))$. Notice also why we didn't decide to reduce the type and define $\Box$ in term of $\to$. In many logics equipped with a strict implication $\strictif$, the necessity operator is definable as $\Box\phi:=(\phi\strictif\phi)\strictif\phi$. In order to respect the content relation required by the semantic of $\to$, a sensible candidate could be $\Box\phi:=(\phi\lor\neg\phi)\to\phi$. The problem is that while the extensional values of the two formulae coincide, that is not the case for their contents: $s( N(\Box\phi))=s( N(\phi))$ and $s( N(\phi\lor\neg\phi)\to\phi)=s( N(\phi)\gru N(\phi))$, which in general are different. 

To see that it is impossible to express $\Box$ with $\to$ plus the Boolean base of connectives, observe that by the definition of $N$ the translation of any formula $\phi$  is a join of elements either of the form $N(p)$, for $p\in Var$, or of the form $N(\phi)\gru N(\psi)$. We can omit either ordering and repetitions, since we are working in a semilattice. Therefore any formula in which $\to$ occurs will contain in a component of the join of the latter form. Now, let $\phi$ be a formula with no occurrences of $\to$. We have $N(\Box\phi)=N(\phi)$, and in particular no element of the form $N(\psi)\gru N(\chi)$ appears in the join $N(\phi)$. Suppose $\Box$ is a definable connective, then it must be definable by a formula $\phi^\to$ containing $\to$ (otherwise their extensional values would differ, by standard modal logic), then some $N(\psi)\gru N(\chi)$ appears in $N(\phi^\to)$. Since no conditions are imposed over $\gru$, a countermodel can be constructed where $N(\psi)\gru N(\chi)$ is assigned a value s.t. $N(\phi)\neq N(\phi^\to)$ (e.g., taking a model where the algebra of contents is the integer line and putting $N(\psi)\gru N(\chi)> N(\phi)$).

The class of $\ModelPaizero$-models characterizes the logic $\Paizero=\pair{\lang_\mathbf{ML}^\to,\vdash_{\Paizero}}$, as we are going to prove. $\Paizero$ is a sublogic of the \textit{global} companion of Parry's logic of analytic implication and it can be presented by the following Hilbert-style system:

\begin{itemize}
    \item[] \textbf{Axioms}
    \item[(A1)] $\phi\supset(\psi\supset\phi)$
    \item[(A2)] $(\phi\supset(\psi\supset\chi))\supset((\phi\supset\psi)\supset(\phi\supset\chi))$
    \item[(A3)] $(\neg\phi\supset\neg\psi)\supset(\psi\supset\phi)$
    \item[(A4)] $(\phi\to\psi)\equiv(\Box(\phi\supset\psi)\land(\psi\prec\phi))$
    \item[(K)] $\Box(\phi\supset\psi)\supset(\Box\phi\supset\Box\psi)$
    \item[(T)] $\Box\phi\supset\phi$
    \item[(4)] $\Box\phi\supset\Box\Box\phi$
    \item[(O1)] $\phi\prec\phi$
    \item[(O2)] $((\phi\prec\psi)\land(\psi\prec\chi))\supset(\phi\prec\chi)$
    \item[(O3)] $(\phi\prec\psi)\supset(\phi\lor\psi\prec\psi)$
    \item[(O4)] $((\phi\prec\psi)\equiv(\phi\prec\neg\psi))\land((\phi\prec\psi)\equiv(\neg\phi\prec\psi))$
    \item[(O5)] $((\phi\prec\psi)\equiv(\phi\prec\Box\psi))\land((\phi\prec\psi)\equiv(\Box\phi\prec\psi))$
    \item[(C1)] $(\phi\prec\phi\lor\psi)\land(\psi\prec\phi\lor\psi)$
    \item[(C2)] $((\phi\prec\chi)\land(\psi\prec\chi))\supset((\phi\lor\psi)\prec\chi)$
    \item[(C3)] $((\phi\prec\chi)\land(\chi\prec\phi)\land(\psi\prec\zeta)\land(\zeta\prec\psi))\supset((\phi\to\psi)\prec(\chi\to\zeta))$
    \item[] \textbf{Rules}
    \item[(MP)] $\phi,\phi\supset\psi\vdash\psi$
    \item[(Nec$_g$)] $\phi\vdash\Box\phi$
\end{itemize}

\n Observe that (A1-3), (K), (T), (4), (MP), (Nec$_g$) provides a complete axiomatization of the global modal logic $\Sfourglobal$. Moreover, using (A4) and (T) it is derivable \textit{modus ponens} for the intensional arrow: $\phi,\phi\to\psi\vdash\psi$. As usual, for a logic $\logic$ by a $\logic$-theory we mean a set $X$ closed under the relation $\vdash_\logic$. When the underlying logic is clear, we omit the prefix and simply call it a theory.

In order to provide a canonical model for the completeness proof, we need to separately build its two main components, the algebra of truth-values and that of contents. Both the algebras will be built from a certain theory. We start with the algebra of contents.

To make the nomenclature simpler to read, in the following by $Fm$ we mean $\Fmobject$. Let $\Gamma\subseteq Fm$ be a theory of $\Paizero$. Let us consider the formula algebra $\Fm_\pair{\oplus,\gru}$ and the translation function of any $\ModelPaizero$-model $ N$, which is surjective, therefore $Fm_\pair{\oplus,\gru} =  N[Fm]$. For $\phi,\psi\in Fm$, we define the relation over $ N[Fm]$: $ N(\phi)\leq_\Gamma N(\psi) \Leftrightarrow \phi\prec\psi\in\Gamma$. Due to negation and box transparency, many formulae have the same translation through $ N$ when negation or box are involved, nonetheless $\leq_\Gamma$ is well-defined, thanks to (O4-5). By (O1-2), $\leq_\Gamma$ is a preorder over $ N[Fm]$.

Next we define: $ N(\phi)\sim_\Gamma N(\psi) \Leftrightarrow \phi\leq_\Gamma\psi$ and $\psi\leq_\Gamma\phi$. This relation is an equivalence, moreover it is a congruence.

\begin{itemize}
    \item[($\oplus$)] Let $ N(\phi_1)\sim_\Gamma N(\psi_1)$ and $ N(\phi_2)\sim_\Gamma N(\psi_2)$. In particular $ N(\phi_1)\leq_\Gamma N(\psi_1)$, so by definition $\phi_1\prec\psi_1\in\Gamma$. $\Gamma$ is a theory, so $\psi_1\prec(\psi_1\lor\psi_2)\in\Gamma$ by (C1). Similarly, from $ N(\phi_2)\leq_\Gamma N(\psi_2)$ it follows $\phi_2\prec(\psi_1\lor\psi_2)\in\Gamma$. By (C2) we have $(\phi_1\lor\phi_2)\prec(\psi_1\lor\psi_2)\in\Gamma$, therefore $ N(\phi_1)\oplus N(\phi_2)\leq_\Gamma N(\psi_1)\oplus N(\psi_2)$. By a similar reasoning, from $ N(\psi_1)\leq_\Gamma N(\phi_1)$ and $ N(\psi_2)\leq_\Gamma N(\phi_2)$ we obtain $ N(\psi_1)\oplus N(\psi_2)\leq_\Gamma N(\phi_1)\oplus N(\phi_2)$. We conclude $ N(\phi_1)\oplus N(\phi_2)\sim_\Gamma N(\psi_1)\oplus N(\psi_2)$.

    \item[($\gru$)] Let $ N(\phi_1)\sim_\Gamma N(\psi_1)$ and $ N(\phi_2)\sim_\Gamma N(\psi_2)$. Since $\phi_1\prec\psi_1,\psi_1\prec\phi_1,\phi_2\prec\psi_2,\psi_2\prec\phi_2\in\Gamma$, by (C3) $(\phi_1\to\phi_2)\prec(\psi_1\to\psi_2), (\psi_1\to\psi_2)\prec(\phi_1\to\phi_2)\in\Gamma$, hence $ N(\phi_1)\gru N(\phi_2)\sim_\Gamma N(\psi_1)\gru N(\psi_2)$.
\end{itemize}

\n The compatibility with operations allows us to obtain the quotient algebra $\langle N[Fm]/$$\sim_\Gamma,\oplus,\gru\rangle$.

\begin{lemma}\label{lemma: pai0_partial_order}
    Let $\Gamma$ be a theory. For $\phi,\psi\in Fm$, it holds: $[ N(\phi)]_{\sim_\Gamma}\leq_\oplus[ N(\psi)]_{\sim_\Gamma}\Leftrightarrow N(\phi)\leq_\Gamma N(\psi)$.
\end{lemma}
\begin{proof}
    ($\Rightarrow$) Let $[ N(\phi)]_{\sim_\Gamma}\leq_\oplus[ N(\psi)]_{\sim_\Gamma}$, that is $[ N(\psi)]_{\sim_\Gamma}=[ N(\phi)]_{\sim_\Gamma}\oplus[ N(\psi)]_{\sim_\Gamma}=[ N(\phi)\oplus N(\psi)]_{\sim_\Gamma}$. Hence $ N(\psi)\sim_\Gamma N(\phi)\oplus N(\psi)= N(\phi\lor\psi)$. This implies that $(\phi\lor\psi)\prec\psi\in\Gamma$. Since by (C1) $\phi\prec(\phi\lor\psi)\in\Gamma$, then $\phi\prec\psi\in\Gamma$ by (O2). We conclude $ N(\phi)\leq_\Gamma N(\psi)$.

    \vspace{0.2cm}

    ($\Leftarrow$) Let $ N(\phi)\leq_\Gamma N(\psi)$, so $\phi\prec\psi\in\Gamma$. Using (O3) $(\phi\lor\psi)\prec\psi\in\Gamma$, by (C1) $\psi\prec(\phi\lor\psi)\in\Gamma$. Hence $ N(\psi)\sim_\Gamma N(\phi)\oplus N(\psi)$, therefore $[ N(\psi)]_{\sim_\Gamma}=[ N(\phi)\oplus N(\psi)]_{\sim_\Gamma}=[ N(\phi)]_{\sim_\Gamma}\oplus[ N(\psi)]_{\sim_\Gamma}$, that is $[ N(\phi)]_{\sim_\Gamma}\leq_\oplus[ N(\psi)]_{\sim_\Gamma}$.
\end{proof}

The previous lemma proves that $\leq_\oplus$ is a partial order, since it inherits reflexivity and transitivity from $\leq_\Gamma$. Antisymmetry follows from definition of $\sim_\Gamma$: if $[ N(\phi)]_{\sim_\Gamma}\leq_\oplus[ N(\psi)]_{\sim_\Gamma}$ and $[ N(\psi)]_{\sim_\Gamma}\leq_\oplus[ N(\phi)]_{\sim_\Gamma}$, then $ N(\phi)\leq_\Gamma N(\psi)$ and $ N(\psi)\leq_\Gamma N(\phi)$, hence $ N(\phi)\sim_\Gamma N(\psi)$.

We conclude that the quotient algebra $\langle N[Fm]/$$\sim_\Gamma,\oplus,\gru\rangle$ is actually a suitable algebra of contents for a $\ModelPaizero$-model, in fact $\langle N[Fm]/$$\sim_\Gamma,\oplus\rangle$ is now join-semilattice. We are going to make use of the following consequence of lemma \ref{lemma: pai0_partial_order}:

\begin{corollary}\label{cor.: pai0_order_correspondence}
    Let $\Gamma$ be a theory. For $\phi,\psi\in Fm$, it holds: $[ N(\phi)]_{\sim_\Gamma}\leq_\oplus[ N(\psi)]_{\sim_\Gamma}\Leftrightarrow\phi\prec\psi\in\Gamma$.
\end{corollary}

Now we are left with the algebra of truth-values for the canonical model. To build it, we will follow an algebraic completeness proof for the global modal logic $\Sfourglobal$ w.r.t. the class of interior algebras. We recall that $\Sfourglobal$ is (A1-3), (K), (T), (4), (MP), (Nec$_g$). We prove completeness using the so-called Lindenbaum-Tarski process, a standard technique in algebraic logic (see \cite[ch. 2.1]{Font16_handbook}).

We consider the basic modal language $\pair{\neg,\lor,\Box}$ and we denote its formula algebra $\Fm_m$. We take a $\Sfourglobal$-theory $\Gamma$. For $\phi,\psi\in Fm_m$, let us define the relation $\Omega\Gamma$ s.t. $\pair{\phi,\psi}\in\Omega\Gamma\Leftrightarrow\phi\equiv\psi\in\Gamma$. 

$\Omega\Gamma$ is a congruence over $\Fm_m$. Compatibility w.r.t. $\neg$ and $\lor$ is immediate from classical logic. For the remaining connective, suppose $\pair{\phi,\psi}\in\Omega\Gamma$, which means $\phi\equiv\psi\in\Gamma$. By (Nec$_g$) $\Box(\phi\equiv\psi)\in\Gamma$, hence using (K) and (MP) we conclude $\Box\psi\equiv\Box\psi\in\Gamma$, that is $\pair{\Box\phi,\Box\psi}\in\Omega\Gamma$.

With this congruence we obtain the quotient algebra $\Fm_m/\Omega\Gamma$, which is known as the Lindenbaum-Tarski algebra of the theory $\Gamma$. This is an interior algebra, in fact for every equation $\phi\approx\psi$ of the equational base of Boolean algebras, it is easy to prove that $\vdash_{\Sfourglobal}\phi\equiv\psi$. Using the well-known completeness of $\Sfourglobal$ w.r.t. preordered Kripke frames, we can simplify calculations and obtain a straightforward proof of the same fact for the modal equations which constitute the equational base of interior algebras.

\begin{lemma}\label{lemma: S4_lind_tarski_process}
    Let $\Gamma$ be a $\Sfourglobal$-theory and $\Fm_m/\Omega\Gamma$ the corresponding Lindenbaum-Tarski algebra. For all $\phi\in Fm_m$, $\phi\in\Gamma$ iff $[\phi]_{\Omega\Gamma}=1_\Gamma$, with $1_\Gamma$ the top element of $Fm_m/\Omega\Gamma$.
\end{lemma}
\begin{proof}
    First we show that $\phi\in\Gamma$ iff $[\phi]_{\Omega\Gamma}\in[\Gamma]_{\Omega\Gamma}$. For the right-to-left direction, suppose $\phi\in\Gamma$ and $\pair{\phi,\psi}\in\Omega\Gamma$, that is $\phi\equiv\psi\in\Gamma$. By classical logic it follows $\psi\in\Gamma$. The other direction is obvious. Now notice that for $\phi,\psi\in\Gamma$, $[\phi]_{\Omega\Gamma}=[\psi]_{\Omega\Gamma}$. The result is obtained simply applying classical logic to derive $\phi\equiv\psi\in\Gamma$. 
    
    Returning to the main claim, for the left-to-right direction take any classical tautology $\tau$. Since $\Gamma$ is a theory we have $\tau\in\Gamma$, therefore for all $\phi\in\Gamma,[\phi]_{\Omega\Gamma}=[\tau]_{\Omega\Gamma}$. For any formula $\alpha$, it is a classical tautology $\tau\equiv(\tau\lor\alpha)$, therefore $[\tau]_{\Omega\Gamma}=[\tau\lor\alpha]_{\Omega\Gamma}=[\tau]_{\Omega\Gamma}\lor[\alpha]_{\Omega\Gamma}$. Considering the order induced by $\lor$ over $\Fm_m/\Omega\Gamma$, this implies $[\alpha]_{\Omega\Gamma}\leq[\tau]_{\Omega\Gamma}=[\phi]_{\Omega\Gamma}$. This holds for every $\alpha$, therefore $[\phi]_{\Omega\Gamma}$ is the top element of $\Fm_m/\Omega\Gamma$. For the right-to-left direction suppose $[\phi]_{\Omega\Gamma}=1_\Gamma$. Again, $\Gamma$ contains every classical tautology $\tau$ and $[\tau]_{\Omega\Gamma}=1_\Gamma$, so $\phi\equiv\tau\in\Gamma$. Because $\Gamma$ is a theory we conclude $\phi\in\Gamma$.
\end{proof}

As a side result, completeness\footnote{Actually a stronger claim can be made, in fact $\Sfourglobal$ is algebraizable (\cite[p. 86]{Font16_handbook}).} follows immediately:

\begin{theorem}
    $\Sfourglobal$ is complete w.r.t. the class of interior algebras.
\end{theorem}
\begin{proof}
    Soundness is proved by routine checking. For adequacy, by contraposition assume that $\Gamma\nvdash_{\Sfourglobal}\phi$. Extend $\Gamma$ to the least theory $\Delta\supseteq\Gamma$. By soundness, $\phi\notin\Delta$. Consider the interior algebra $\Fm_m/\Omega\Delta$ and the function $v(\phi)=[\phi]_{\Omega\Delta}$. By compatibility of $\Omega\Delta$ with the operations, $v$ is a valuation. By lemma \ref{lemma: S4_lind_tarski_process}, for all $\gamma\in\Delta, v(\gamma)=1_{\Delta}$, while $v(\phi)\neq 1_{\Delta}$.
\end{proof}

We have all the tools to prove the completeness of $\Paizero$. We will employ a canonical model strategy. First we need some definitions.

\begin{definition}\label{def.: consistet_and_prime}
    Let $\logic$ be a logic for a language $\tau$. A set of formulae $\Gamma\subseteq Fm_\tau$ is $\logic$-consistent if there exists $\phi\in Fm_\tau$ s.t. $\Gamma\nvdash_\logic\phi$. In particular, $\Gamma$ is $\logic$-$\phi$-consistent if $\Gamma\nvdash_\logic\phi$. For a language with a join operation, $\Gamma\neq Fm_\tau$ is prime if $\Gamma\vdash_{\logic}\phi\lor\psi$ implies $\Gamma\vdash_{\logic}\phi$ or $\Gamma\vdash_{\logic}\psi$. 
\end{definition}

\n When the context is clear, we might omit the prefix and simply write consistent or $\phi$-consistent. Observe that every prime theory is consistent.

In order to obtain a suitable result playing the role of Lindenbaum's lemma, we are going to perform a more abstract algebraic reasoning and conclude general lattice-theoretic results for closure operators. To the best of my knowledge, this strategy has been devised by \cite{Paoli_Prenosil_st}. As a corollary, the standard Lindenbaum's lemma for classical logic and some of its expansions is regained.

Let us recall some basic definitions from lattice theory. A lattice $\textbf{L}=\pair{L,\lor,\land}$ is \textit{complete} if meets and joins exist for arbitrary subsets of $L$. An element $a\in L$ of a complete lattice is called \textit{compact} if for any $S\subseteq L, a \leq \bigvee S$ implies that there exists finite $S'\subseteq S$ such that $a \leq \bigvee S'$. A complete lattice is \textit{algebraic} if every element can be expressed as the join of compact elements. In a complete lattice, an element $a\in \textbf{L}$ is \textit{completely meet-irreducible} when, for arbitrary $S\subseteq L$, from $a=\bigwedge S$ it follows $a\in S$. An element is simply \textit{meet-irreducible} if the previous property holds for finite $S$. The following result is a straightforward application of Zorn's lemma:

\begin{lemma}[\cite{Metcalf_Paoli_Tsinakis_residuated_book}, proposition 1.3.8]\label{lemma: algebraic_lattice_meet_irred_expansion}
    Let $\Le$ be an algebraic lattice, $a,c\in L$ with $c$ a compact element s.t. $c\nleq a$. There exists a completely meet-irreducible element $a^+\in L$ s.t. $a\leq a^+$ and $c\nleq a$.
\end{lemma}

Since a logic $\logic$ is a closure relation $\vdash_\logic$ over formulae $Fm$ of a certain language, it can equally be seen as a \textit{closure operator} $\C^\logic:\mathscr{P}(Fm)\to Fm$ s.t. $\C_\logic(\Gamma)=\{ \phi$ $|$ $\Gamma\vdash_{\logic}\phi \}$. Each closure operator has associated a \textit{closure system} $\mathsf{Th}(\logic)\subseteq\mathscr{P}(Fm)$, that is a collection of subsets satisfying $Fm\in\mathsf{Th}(\logic)$ and closed under arbitrary non-empty intersections, defined as $\mathsf{Th}(\logic)=\{ \Gamma$ $|$ $\C_\logic(\Gamma)=\Gamma \}$, which forms the lattice of all $\logic$-theories. A logic $\logic$ is \textit{finitary} when $\Gamma\vdash_{\logic}\phi$ implies $\Gamma'\vdash_{\logic}\phi$ for finite $\Gamma'\subseteq\Gamma$. The closure operator $\C_\logic$ of a finitary logic is finitary as well, in the sense that $\C_\logic(\Gamma)=\bigcup\{ \C_\logic(\Gamma')$ $|$ finite $\Gamma'\subseteq\Gamma \}$. $\Gamma\in\mathsf{Th}(\logic)$ is \textit{finitely generated} if $\Gamma=\C_\logic(\Gamma')$ for finite $\Gamma'$.

\begin{lemma}[\cite{Font16_handbook}, propositions 1.65-1.67]\label{lemma: finitary_to_algebraic}
    If $\C$ is a finitary closure operator, its associated closure system $\mathscr{C}$ is an algebraic lattice. The compact elements of $\mathscr{C}$ are precisely the finitely generated ones.
\end{lemma}

\n All the logics considered throughout this paper have finite Hilbert-style presentations, therefore they are finitary, hence their lattices of theories are algebraic. The following results apply to all of our logics as well.

\begin{lemma}\label{lemma: meet_irred_to_prime_theory}
    Let the logic $\logic$ be finitary and a conservative expansion of classical logic. If $\Gamma$ is a completely meet-irreducible theory then it is a prime theory.
\end{lemma}
\begin{proof}
    First observe that the following hold: $\Gamma,\phi\lor\psi\vdash\chi\Leftrightarrow\Gamma,\phi\vdash\chi$ and $\Gamma,\psi\vdash\chi$. For the left-to-right direction, assume that $\Gamma,\phi\lor\psi\vdash\chi$, then by classical logic $\Gamma,\phi\vdash\phi\lor\psi$, hence applying the hypothesis $\Gamma,\phi\vdash\chi$. The same can be repeated substituting $\psi$ for $\phi$, obtaining $\Gamma,\psi\vdash\chi$. For the right-to-left direction, suppose $\Gamma,\phi\vdash\chi$ and $\Gamma,\psi\vdash\chi$. Since we can use the entirety of classical logic, we can borrow the rule of $\lor$-elimination from natural deduction and immediately obtain $\Gamma,\phi\lor\psi\vdash\chi$.
    

    \n Returning to the main lemma, let $\Gamma$ be a completely meet-irreducible theory in $\mathsf{Th}(\logic)$ and suppose $\Gamma\vdash\phi\lor\psi$. Since $\Gamma$ is a theory, $\Gamma=\C_{\logic}(\Gamma,\phi\lor\psi)$, moreover by the above point also $\C_{\logic}(\Gamma,\phi\lor\psi)=\C_{\logic}(\Gamma,\phi)\cap\C_{\logic}(\Gamma,\psi)$. Since $\Gamma$ is meet-irreducible, we conclude $\Gamma=\C_{\logic}(\Gamma,\phi)$ or $\Gamma=\C_{\logic}(\Gamma,\psi)$, that is $\Gamma\vdash\phi$ or $\Gamma\vdash\psi$.
\end{proof}

Notice that in the lattice of theories of a logic, $\Gamma\nvdash\phi$ amounts to $\C_{\logic}(\phi)\nsubseteq\C_{\logic}(\Gamma)$. The next theorem will play the role of Lindenbaum's lemma:

\begin{theorem}\label{th.: lindenbaum_like_lemma}
    Let the logic $\logic$ be finitary and a conservative expansion of classical logic. If $\Gamma$ is a theory s.t. $\Gamma\nvdash_\logic\phi$, there exists a prime theory $\Delta\supseteq\Gamma$ s.t. $\Delta\nvdash_\logic\phi$.
\end{theorem}
\begin{proof}
    Since $\logic$ is finitary, its lattice of theories $\mathsf{Th}(\logic)$ is algebraic, as stated in lemma \ref{lemma: finitary_to_algebraic}, moreover $\C_{\logic}(\phi)$ is compact since it is finitely generated. $\Gamma$ is a theory, which means $\C_{\logic}(\Gamma)=\Gamma$, and $\Gamma\nvdash\phi$ is equivalent to $\C_{\logic}(\phi)\nsubseteq\Gamma$. Applying lemma \ref{lemma: algebraic_lattice_meet_irred_expansion}, there exists a completely meet-irreducible theory $\Delta\supseteq\Gamma$ s.t. $\C_{\logic}(\phi)\nsubseteq\Delta$, that is $\Delta\nvdash\phi$. By lemma \ref{lemma: meet_irred_to_prime_theory}, $\Delta$ is a prime theory.
\end{proof}

\begin{definition}\label{def.: pai0_canonical_model}
    Let $\Gamma$ be a prime $\Paizero$-theory. The canonical model for $\Gamma$ is the $\ModelPaizero$-model $\M^\Gamma=\langle\Fm_m/\Omega v_1[\Gamma],\langle N[Fm]/$$\sim_\Gamma,\oplus,\gru\rangle, N,v^\Gamma,\pi\rangle$, with:
    \begin{itemize}
        \item $v^\Gamma=v_2\circ v_1$, with $v_1:Fm\to Fm_m$ the identity function w.r.t. to variables, a homomorphism w.r.t. $\neg,\lor,\Box$ and:

        $$v_1(\phi\to\psi)=\begin{cases} 
	\Box v_1(\neg\phi\lor\psi) & \text{if } [ N(\psi)]_{\sim_\Gamma}\leq_\oplus[ N(\phi)]_{\sim_\Gamma};\\
        v_1(\neg\phi\lor\psi)\land\neg v_1(\neg\phi\lor\psi) & \text{otherwise.}
        \end{cases}$$


        $v_2:Fm_m\to Fm_m/\Omega v_1[\Gamma]$ is defined as $v_2(\phi)=[\phi]_{\Omega v_1[\Gamma]}$.
        
        \item for all $\phi\in N(Fm),\pi(\phi)=[\phi]_{\sim_\Gamma}$.
    \end{itemize}
\end{definition}

\n Observe that $v^\Gamma$ obeys the conditions for a valuation of a $\ModelPaizero$-model, in fact if $[N(\psi)]_{\sim_\Gamma}\nleq_\oplus[ N(\phi)]_{\sim_\Gamma}$, $v^\Gamma(\phi\to\psi)=[v_1(\neg\phi\lor\psi)]_{\Omega v_1[\Gamma]}\land\neg[v_1(\neg\phi\lor\psi)]_{\Omega v_1[\Gamma]}=0_\Gamma$, with $0_\Gamma$ the bottom element of the algebra, because $\Fm_m/\Omega v_1[\Gamma]$ is also a Boolean algebra.

\begin{lemma}\label{lemma: pai0_prime_to_S4_prime}
    Let $\Gamma$ be a prime $\Paizero$-theory and $\M^\Gamma$ its canonical model. Then $v_1[\Gamma]$ is a prime $\Sfourglobal$-theory.
\end{lemma}
\begin{proof}
    That $v_1[\Gamma]$ is a $\Sfourglobal$-theory follows from (A1-3), (K), (T), (4) all belongs to $\Gamma$ and that it is closed under (MP) and (Nec$_g$). For closure under (MP), let $v_1(\phi), v_1(\phi\supset\psi)\in v_1[\Gamma]$, then $\phi,\phi\supset\psi\in\Gamma$. It follows $\psi\in\Gamma$, so $v_1(\psi)\in v_1[\Gamma]$. For (Nec$_g$), let $v_1(\phi)$ so $\phi\in\Gamma$, hence $\Box\phi\in\Gamma$ and $v_1(\Box\phi) = \Box v_1(\psi)\in v_1[\Gamma]$.
    
    The consistency of $v_1[\Gamma]$ follows from that of $\Gamma$. In fact if $v_1[\Gamma]$ were inconsistent then, by classical logic, $p\land\neg p\in v_1[\Gamma]$ for every variable $p$. Since $v_1[\Gamma]$ is a $\Sfourglobal$-theory, $p\in v_1[\Gamma]$ and $\neg p\in v_1[\Gamma]$. By definition \ref{def.: pai0_canonical_model}, $v_1$  is the identity w.r.t. to variables, therefore the only possibility for the preimage of $p$ is that $p=v_1(p)$, therefore $p\in v_1[\Gamma]$ implies $p\in\Gamma$, moreover $\neg p=\neg v_1(p)=v_1(\neg p)\in v_1[\Gamma]$ implies $\neg p\in\Gamma$, which contradicts the consistency of $\Gamma$. Observe that it cannot be the case that $v_1(\neg\phi'\lor\psi')\land\neg v_1(\neg\phi'\lor\psi')\in v_1[\Gamma]$ for some $\phi'\to\psi'\in\Gamma$ s.t. $[ N(\psi')]_{\sim_\Gamma}\nleq_\oplus[ N(\phi')]_{\sim_\Gamma}$. In fact, by corollary \ref{cor.: pai0_order_correspondence}, this would imply $\psi'\prec\phi'\notin\Gamma$, which by (A4) forces $\phi'\to\psi'\notin\Gamma$.

    For primality, suppose $\phi\lor\psi\in v_1[\Gamma]$. Among the possible preimages for $\phi\lor\psi$ through $v_1$ there is itself, and since $\phi\lor\psi\in Fm_m$ we have, by definition of $v_1$, $v_1(\phi\lor\psi)=v_1(\phi)\lor v_1(\psi)=\phi\lor\psi$, so $v_1(\phi)=\phi$ and $v_1(\psi)=\psi$. $\Gamma$ is prime, so from $\phi\lor\psi\in\Gamma$ follow either $\phi\in\Gamma$ or $\psi\in\Gamma$, therefore $v_1(\phi)=\phi\in v_1[\Gamma]$ or $v_1(\psi)=\psi\in v_1[\Gamma]$.
\end{proof}

\n Since $v_1[\Gamma]$ is a $\Sfourglobal$-theory, we can obtain the Lindenbaum-Tarski algebra $\Fm_m/\Omega v_1[\Gamma]$, that is an interior algebra. 

\begin{lemma}\label{lemma: pai0_bridge_for_canonical_valuation}
    Let $\Gamma$ be a prime $\Paizero$-theory and $\M^\Gamma$ its canonical model. Then for all $\phi\in Fm$, $\phi\in\Gamma\Leftrightarrow v_1(\phi)\in v_1[\Gamma]$.
\end{lemma}
\begin{proof}
    Recall that, by lemma \ref{lemma: pai0_prime_to_S4_prime}, $v_1[\Gamma]$ is prime $\Sfourglobal$-theory. We only consider the right-to-left direction, the opposite is trivial. We proceed by induction on the complexity of $\phi$. For the base case, if $\phi:=p$ for $p\in Var$, by definition $v_1$ is the identity over variables, therefore $v_1(p)=p\in v_1[\Gamma]$. Hence the only preimage of $p$ is $p$ itself, that is $p\in\Gamma$.

    Assume that the property holds for formulae simpler than $\phi$. If $\phi:=\alpha\lor\beta$ and $v_1(\alpha\lor\beta)=v_1(\alpha)\lor v_1(\beta)\in v_1[\Gamma]$. By primality, $v_1(\alpha)\in v_1[\Gamma]$ or $v_1(\beta)\in v_1[\Gamma]$. By induction hypothesis $\alpha\in\Gamma$ or $\beta\in\Gamma$, hence by disjunction introduction $\alpha\lor\beta\in\Gamma$.
    
    Suppose $\phi:=\neg\alpha$ and $v_1(\neg\alpha)=\neg v_1(\alpha)\in v_1[\Gamma]$. By consistency, $v_1(\alpha)\notin v_1[\Gamma]$. By induction hypothesis $\alpha\notin\Gamma$. $\Gamma$ contains all classical tautologies, hence $\alpha\lor\neg\alpha\in\Gamma$, and by primality either $\alpha\in\Gamma$ or $\neg\alpha\in\Gamma$, then $\neg\alpha\in\Gamma$.

    Let $\phi:=\Box\alpha$ and $v_1(\Box\alpha)=\Box v_1(\alpha)\in v_1[\Gamma]$. $v_1[\Gamma]$ is a $\Sfourglobal$-theory, hence by (T) $v_1(\alpha)\in v_1[\Gamma]$. By induction hypothesis $\alpha\in\Gamma$, and by (Nec$_g$) $\Box\alpha\in\Gamma$.

    Suppose $\phi:=\alpha\to\beta$ and $v_1(\alpha\to\beta)\in v_1[\Gamma]$. We have to consider two cases. (1) If in the canonical model it holds $[ N(\psi)]_{\sim_\Gamma}\leq_\oplus[ N(\phi)]_{\sim_\Gamma}$, then $v_1(\alpha\to\beta)=\Box v_1(\neg\alpha\lor\beta)\in v_1[\Gamma]$. By (T) $v_1(\neg\alpha\lor\beta)\in v_1[\Gamma]$. We can use the previous cases for disjunction and negation to apply the induction hypothesis and obtain $\neg\alpha\lor\beta\in\Gamma$, then by (Nec$_g$) $\Box(\neg\alpha\lor\beta)\in\Gamma$. By lemma \ref{cor.: pai0_order_correspondence}, from $[ N(\psi)]_{\sim_\Gamma}\leq_\oplus[ N(\phi)]_{\sim_\Gamma}$ we have $\psi\prec\phi\in\Gamma$. By (A4) $\phi\to\psi\in\Gamma$. (2) If in the canonical model it holds $[ N(\psi)]_{\sim_\Gamma}\nleq_\oplus[ N(\phi)]_{\sim_\Gamma}$, then $v_1(\neg\phi\lor\psi)\land\neg v_1(\neg\phi\lor\psi)\in v_1[\Gamma]$. This cannot happen, since $v_1[\Gamma]$ is consistent by lemma \ref{lemma: pai0_prime_to_S4_prime}. Therefore the second case is discarded.
\end{proof}

\begin{lemma}\label{lemma: pai0_canonical_valuation}
    Let $\Gamma$ be a prime $\Paizero$-theory and $\M^\Gamma$ its canonical model. For all $\phi\in Fm$, $v^\Gamma(\phi)=1_{v_1[\Gamma]}\Leftrightarrow\phi\in\Gamma$.
\end{lemma}
\begin{proof}
    Let us denote by $\tau$ an arbitrary classical tautology. Then $v^\Gamma(\phi)=[v_1(\phi)]_{\Omega v_1[\Gamma]}=1_{v_1[\Gamma]}=[\tau]_{\Omega v_1[\Gamma]}$  (the last identity holds because $1_{v_1[\Gamma]}$ is the top element of $Fm_m/\Omega v_1[\Gamma]$ and contains every theorem) iff $v_1(\phi)\equiv \tau\in v_1[\Gamma]$ iff, by classical logic, $v_1(\phi)\in v_1[\Gamma]$ iff, by lemma \ref{lemma: pai0_bridge_for_canonical_valuation}, $\phi\in\Gamma$. 
\end{proof}


\begin{theorem}\label{th.: pai0_completeness}
    $\Paizero$ is complete w.r.t. the class of $\ModelPaizero$-models.
\end{theorem}
\begin{proof}
    Soundness is proved by routine checking; in particular, all the axioms containing $\prec$ can be easily checked by the correspondence between $v(\alpha\prec\beta)=1$ and the fact that $N(s(\alpha))\leq_\oplus N(s(\beta))$ holds in the relative semilattice. For adequacy, by contraposition assume that $\Gamma\nvdash_{\Paizero}\phi$. Using theorem \ref{th.: lindenbaum_like_lemma}, extend $\Gamma$ to a prime $\Paizero$-theory $\Delta\supseteq\Gamma$ s.t. $\phi\notin\Delta$. Consider the canonical model for $\Delta$ of definition \ref{def.: pai0_canonical_model}. By lemma \ref{lemma: pai0_canonical_valuation}, for all $\gamma\in\Delta, v^\Gamma(\gamma)=1_{v_1[\Delta]}$, which holds in particular for all $\gamma\in\Gamma$, while $v^\Gamma(\phi)\neq 1_{v_1[\Delta]}$.
\end{proof}

In order to strengthen $\Paizero$ and recover Parry's logic of analytic implication (in its global version), it is enough to extend the system with:

\begin{itemize}
    \item[(A5)] $((\phi\to\psi)\prec(\phi\lor\psi))\land((\phi\lor\psi)\prec(\phi\to\psi))$
\end{itemize}

\n We define $\Paifull=\pair{\lang_\mathbf{ML}^\to,\vdash_{\Paifull}}=\Paizero+(A5)$, where by $\logic$ + (Ax$_1$),... (Ax$_n$) we mean the axiomatic extension of logic $\logic$ by formulae (Ax$_1$),... (Ax$_n$).

As (A5) suggests, what we want in the semantics for $\Paifull$ is for the content of $\phi\to\psi$ to coincide with that of $\phi\lor\psi$, in other words we want the $\gru$ operation to collapse on $\oplus$. This motivates the following definition: 

\begin{definition}\label{def.: pai-model}
    A $\ModelPaifull$-model is a $\ModelPaizero$-model $\pair{\B,\pair{S,\oplus,\gru}, N,v,s}$ which satisfies the further condition $\gru=\oplus$. Equivalently, a $\ModelPaifull$-model is a structure $\pair{\B,\pair{S,\oplus}, N,v,s}$ s.t. $ N(\phi\to\psi)= N(\phi)\oplus N(\psi)$, while the remaining components are defined adapting definition \ref{def.: pai0-model}.
\end{definition}

\n We will stick to the latter definition. The completeness of $\Paifull$ is an adaptation of that of $\Paizero$ that starts with lemma \ref{lemma: pai0_partial_order} and ends with theorem \ref{th.: pai0_completeness}. We make just a few comments.

The relation $ N(\phi)\sim_\Gamma N(\psi)$ defined over $ N(Fm)$ is now used to obtain an algebra of contents $\langle N[Fm]/$$\sim_\Gamma,\oplus\rangle$. Due to the simpler type, there is no need to consider the compatibility of the equivalence w.r.t. $\gru$, therefore axiom (C3) is no longer used. Lemma \ref{lemma: pai0_partial_order} is modified in the following way:

\begin{lemma}\label{lemma: paifull_partial_order}
    Let $\Gamma$ be a $\Paifull$-theory. For $\phi,\psi\in Fm$, it holds: $[ N(\phi)]_{\sim_\Gamma}\leq_\oplus[ N(\psi)]_{\sim_\Gamma}\Leftrightarrow N(\phi)\leq_\Gamma N(\psi)$.
\end{lemma}
\begin{proof}
    For the left-to-right direction, let $[ N(\phi)]_{\sim_\Gamma}\leq_\oplus[ N(\psi)]_{\sim_\Gamma}$, that is $[ N(\psi)]_{\sim_\Gamma}=[ N(\phi)]_{\sim_\Gamma}\oplus[ N(\psi)]_{\sim_\Gamma}=[ N(\phi)\oplus N(\psi)]_{\sim_\Gamma}$. Hence $ N(\psi)\sim_\Gamma N(\phi)\oplus N(\psi)$. Now $ N(\phi)\oplus N(\psi)$ can have two preimages: either $\phi\lor\psi$ or $\phi\to\psi$. If $ N(\phi)\oplus N(\psi)= N(\phi\lor\psi)$ the proof is the same as in lemma \ref{lemma: pai0_partial_order}. If $ N(\phi)\oplus N(\psi)= N(\phi\to\psi)$ then $(\phi\to\psi)\prec\psi$, by (A5) $(\phi\lor\psi)\prec(\phi\to\psi)$, and by (O2) $(\phi\lor\psi)\prec\psi$. Using (C1) $\phi\prec(\phi\lor\psi)\in\Gamma$, then $\phi\prec\psi\in\Gamma$ by (O2). We conclude $ N(\phi)\leq_\Gamma N(\psi)$. The remaining of the proof follows lemma \ref{lemma: pai0_partial_order}.
\end{proof}

\n With this, the reasoning which leads to corollary \ref{cor.: pai0_order_correspondence} can be repeated for $\Paifull$:

\begin{corollary}\label{cor.: paifull_order_correspondence}
    Let $\Gamma$ be a $\Paifull$-theory. For $\phi,\psi\in Fm$, it holds: $[ N(\phi)]_{\sim_\Gamma}\leq_\oplus[ N(\psi)]_{\sim_\Gamma}\Leftrightarrow\phi\prec\psi\in\Gamma$.
\end{corollary}

The canonical model for $\Paifull$ is an adaptation of the one for $\Paizero$:

\begin{definition}\label{def.: paifull_canonical_model}
    Let $\Gamma$ be a prime $\Paifull$-theory. The canonical model for $\Gamma$ is the $\ModelPaifull$-model $\M^\Gamma=\langle\Fm_m/\Omega v_1[\Gamma],\langle N[Fm]/$$\sim_\Gamma,\oplus\rangle, N,v^\Gamma,\pi\rangle$, with $v^\Gamma=v_2\circ v_1$ and $\pi$ defined as in definition \ref{def.: pai0_canonical_model}.
\end{definition}

\begin{theorem}\label{th.: paifull_completeness}
    $\Paifull$ is complete w.r.t. the class of $\ModelPaifull$-models.
\end{theorem}
\begin{proof}
    Assume $\Gamma\nvdash\phi$ and use theorem \ref{th.: lindenbaum_like_lemma} to extend $\Gamma$ to a prime $\Paifull$-theory $\Delta$ s.t. $\phi\notin\Delta$. This can be done since the mentioned theorem is a general algebraic fact that holds for $\Paifull$ as well. The proof then is the same as theorem \ref{th.: pai0_completeness}, using the canonical model for $\Delta$ per definition \ref{def.: paifull_canonical_model} as countermodel.
\end{proof}

Recall that in the completeness proof axiom (C3) played no role. We can say more and remove (O5), that is box transparency, as well. Using the above completeness result we can simplify the language, observing that $\Box$ is logically equivalent to a term using the remaining connectives of the language, overcoming the issue arisen with $\Paizero$. We introduce the abbreviation $\tau_\phi:=\phi\lor\neg\phi$, in order to express the necessity modality operator as $\tau_\phi\to\phi$. In fact since $ N(\phi\to\psi)= N(\phi\lor\psi)$ in a $\ModelPaifull$-model, $\Box\phi$ and $\tau_\phi\to\phi$ coincide both on extensional and intensional values. The previous axiomatic system for $\Paifull$ can be reformulated in the language $\pair{\neg,\lor,\to}$.


That (O5) can be derived is easy to check. By (C1) $\psi\prec((\psi\lor\neg\psi)\lor\psi)$, by (A5) $((\psi\lor\neg\psi)\lor\psi)\prec((\psi\lor\neg\psi)\to\psi)$, by (O2) $\psi\prec((\psi\lor\neg\psi)\to\psi)=(\tau_\psi\to\psi)$. Hence $\phi\prec\psi$ implies, using (O2), $\phi\prec(\tau_\psi\to\psi)$. For the other direction, by (A5) $((\psi\to\neg\psi)\lor\psi)\prec((\psi\lor\neg\psi)\lor\psi)=(\psi\lor\neg\psi)$, where the last identity holds by classical logic. We have $\psi\prec\psi$ and $\psi\prec\neg\psi$, respectively using (O1) and (O4), so by (C2) $(\psi\lor\neg\psi)\prec\psi$. So $(\tau_\psi\to\psi)\prec\psi$, hence from $\phi\prec(\tau_\psi\to\psi)$ we obtain $\phi\prec\psi$. We conclude $(\phi\prec\psi)\equiv(\phi\prec(\tau_\psi\to\psi))$. Similarly for transparency in the left position.

We can finally prove the claim that $\Paifull$ is actually the global companion of Fine's logic of analytic implication $\FineGlobalPai$.

\begin{theorem}\label{th.: uguaglianza_pai_e_fine}
    $\Paifull$ and $\FineGlobalPai$ are the same logic.
\end{theorem}
\begin{proof}
    It is enough to check that $\Paifull$ is sound w.r.t. to Fine models of definition \ref{def.: Fine_PAI_model} (with the relation of logical consequence intended in the global sense), and that $\FineGlobalPai$ is sound w.r.t. $\ModelPaioriginal$-models. Then using the completeness theorems \ref{th.: pai_global_kripke_completeness} and \ref{th.: paifull_completeness} the result follows.
\end{proof}

In order to appreciate the relation between the sublogic $\Paizero$ and $\FergusonPai$, we need to modify the language of Ferguson's system to equate the types. In Ferguson's logic, $\Box$ is neither a primitive connective nor it is definable. At the same time, since this $\FergusonPai$ is already complete w.r.t. to a class of augmented $\mathbf{S4}$-frames, extending the language with $\Box$ and adding to the calculus the axioms of $\mathbf{S4}$ and the (local) necessitation rule doesn't alter the completeness theorem \cite[th. 3]{Ferguson23_subject_matter1}, which can be replicated for the system $\FergusonPaiBox=\pair{\lang_\mathbf{ML}^\to,\vdash_{\FergusonPaiBox}}=\FergusonPai+$(K)+(T)+(4)+(Nec). Finally, as it was done with Fine's logic in theorem \ref{th.: pai_global_kripke_completeness}, we can adapt Ferguson's completeness proof of $\FergusonPai$ for the global companion $\FergusonGlobalPai=\pair{\lang_\mathbf{ML}^\to,\vdash_{\FergusonGlobalPai}}=\FergusonPai+$(K)+(T)+(4)+(Nec$_g$) w.r.t. the same class of frames of definition \ref{def.: ferguson_ca_pai_model}, interpreting now the logical consequence globally. With these modifications we can adapt the same strategy of theorem \ref{th.: uguaglianza_pai_e_fine}. 

\begin{proposition}\label{th.: uguaglianza_pai_e_ferguson}
    $\Paizero$ and $\FergusonGlobalPai$ are the same logic.
\end{proposition}

\end{section}

\begin{section}{Recapturing local logics of analytic implication}\label{sec: local_Parry_logic}

So far we have been able to apply the new generalized Epstein semantics only to approximate our goal, Parry's (local) logic of analytic implication. Its global companion is in fact a stronger logic compared to Parry's original one, as witnessed by the rule of global necessitation: $\phi\vdash\Box\phi$. With the previous construction employed in completeness theorem \ref{th.: paifull_completeness} we are not able to build countermodels that invalidate such rule. Performing the Lindenbaum-Tarski process for the premise of the rule, the equivalence class of $\phi$ will always coincide with the top element $1$ of the quotient of the formula algebra, but $1$ is a fix-point for $\Box$, because $\Box 1\approx 1$, therefore the equivalence class of $\Box\phi$ coincides with $1$ as well. We conclude that either there are not enough countermodels in the class of $\ModelPaifull$-models to prove completeness or the notion of logical consequence is not the correct one. As we are going to see, it is the latter case.

By observing at algebraic level the relation between local and global modal logics, we can get a hint towards the solution for a complete general Epstein semantics for $\Paioriginal$: in the case of normal modal logic \textbf{K}$^g$, the global logic is complete w.r.t. Boolean algebras with operators when the entailment relation is intended in the standard sense, i.e. the preservation of $1$, or in other words it is the $1$-assertional logic induced by said class of algebras. In order to obtain a complete semantics for local \textbf{K} we don't need to switch or enlarge the class, but the local logic will be the logic of order induced by the class of Boolean algebras with operators.
\begin{definition}
    An ordered algebra is a algebra with a partial order $\leq$ over its support. Let $\class{K}$ be a class of ordered $\tau$-algebras. The \textit{logic of order} (or logic preserving degrees of truth) induced by $\class{K}$ is the consequence relation defined as: $\Gamma\vDash_\class{K}^\leq\phi\Leftrightarrow$ for all $\A\in\class{K}$, homomorphisms $h:Fm_\tau\to A$ and $x\in A$, if $x\leq h(\gamma)$ for each $\gamma\in\Gamma$, then $x\leq h(\phi)$.
\end{definition}

The same holds for the standard extensions of \textbf{K}, like $\mathbf{S4}$: $\Sfourglobal$ is the $1$-assertional logic of interior algebras, $\Sfourlocal$ is the logic of order of, again, interior algebras. We can easily adapt the definition to generalized Epstein semantics:
\begin{definition}
    Let $\class{K}$ be a class of $gE$-models s.t. for each $\pair{\A,\B, N,v_A,v_B}\in\class{K}, \A$ is an ordered algebra. The logic of order induced by $\class{K}$ is the consequence relation defined as: $\Gamma\vDash_\class{K}^\leq\phi\Leftrightarrow$ for all $\A\in\class{K}$ and $x\in A$, if $x\leq v_A(\gamma)$ for each $\gamma\in\Gamma$, then $x\leq v_A(\phi)$.
\end{definition}

We claim that $\Paioriginal$ is the logic of order of $\ModelPaifull$-models. In order to do it we are going to use the Kripke completeness for $\Paioriginal$ of \cite{Fine86_analytic_implication}, showing that giving a countermodel for a non-valid entailment of $\Paioriginal$ we can build a $\ModelPaifull$-model whose logic of order invalidates the same entailment. We start by adapting the notion of generated submodel for Fine's $\ModelPaioriginal$-models, a standard tool in modal invariance results (cf. \cite{Blue_book}, def. 2.5, prop. 2.6).

\begin{definition}\label{def.: generated_submodel}
    A $\ModelPaioriginal$-model $\pair{W',R',\pair{T'_w,\oplus_w}_{w\in W'},v',\{t'_w\}_{w\in W'}}$ is a \textit{generated submodel} of $\pair{W,R,\pair{T_w,\oplus_w}_{w\in W},v,\{t_w\}_{w\in W}}$ if $W'\subseteq W, R'=R\cap W^2, v'(p)=v(p)\cap W'$ for all $p\in Var$, and if $x\in W'$ then $R[x]\subseteq W'$. 
\end{definition}

\begin{lemma}\label{lem.: generated_submodel_satisfiability}
    Let $\M'=\pair{W',R',\pair{T_w,\oplus_w}_{w\in W'},v',\{t_w\}_{w\in W'}}$ be a generated submodel of $\M=\pair{W,R,\pair{T_w,\oplus_w}_{w\in W},v,\{t_w\}_{w\in W}}$. For every $w\in W'$: $\M',w\Vdash\phi\Leftrightarrow\M,w\Vdash\phi$.
\end{lemma}
\begin{proof}
    The prove is a straightforward induction over the complexity of $\phi$. Without going through the proof, it is enough to notice that for $w\in W'$, for each propositional variable $v'(p)=v(p)\cap W'$ and $t_w$ is shared between the two models. Therefore the value and content of each atomic formula is preserved between the models, and, by construction of the submodel, the value of every formula is preserved as well (modal formulae are preserved because $W'$ is required to be upwards closed w.r.t. $R$).
\end{proof}

\begin{lemma}\label{lem: Fine_surjective_models}
    For every $\ModelPaioriginal$-model $\pair{W,R,\pair{T_w,\oplus_w}_{w\in W},v,\{t_w\}_{w\in W}}$ there is an equivalent model $\pair{W,R,\pair{T'_w,\oplus_w}_{w\in W},v,\{t'_w\}_{w\in W}}$ in which each $t'_w$ is surjective. 
\end{lemma}
\begin{proof}
    Consider the model $\M'$$=$$\pair{W,R,\pair{t_w[T_w],\oplus_w}_{w\in W},v,\{t_w\}_{w\in W}}$. Clearly each $t_w$ is surjective. We prove that for all $w\in W. \M,w\Vdash\phi\Leftrightarrow\M',w\Vdash\phi$. The claim obviously holds for $\to$-free formulae. For the remaining case: $\M,w\Vdash\alpha\to\beta$ iff $\M,w\Vdash\Box(\alpha\supset\beta)$ and $t_w(\beta)\leq_{\oplus_w}t_w(\alpha)$ iff, by the previous cases and induction hypothesis, $\M',w\Vdash\Box(\alpha\supset\beta)$ and $t_w(\beta)\leq_{\oplus_w}t_w(\alpha)$ (where the latter holds since we are considering the subsemilattice $t_w[T_w]$) iff $\M',w\Vdash\alpha\to\beta$.
\end{proof}

\n In the following we are going to assume all $\ModelPaioriginal$-models to have surjective topic assignment functions.

The next step is to build a $gE$-model out of a $\ModelPaioriginal$-model. In particular, given a rooted $\ModelPaioriginal$-model, we want a $\ModelPaifull$-model that contains all the information stored in the root of Fine's model. In order to move from frames to algebras, we use a construction employed in the J\'{o}nnson-Tarski duality for Boolean algebras with operators (see \cite[ch. 5]{Blue_book}), complex algebras.

\begin{definition}\label{def. complex_algebra}
    For a set $X$, its powerset algebra is the Boolean algebra $\mathscr{P}(\mathbf{X})=\pair{\mathscr{P}(X),\cap,\cup,-,X,\emptyset}$. Given a Kripke frame $\kripkeframe=\pair{W,R}$, consider the powerset algebra $\mathscr{P}(\mathbf{W})$. For $S\subseteq W$, we define the operation: $\Box_R S:=\{ x\in S$ $|$ $R[x]\subseteq S\}$. The \textit{full complex algebra} of $\kripkeframe$ is the algebra $\pair{\mathscr{P}(\mathbf{W}),\Box_R}$.
\end{definition}

It is well-known that starting from a Kripke frame $\pair{W,R}$ its full complex algebra $\pair{\mathscr{P}(\mathbf{W}),\Box_R}$ is a Boolean algebra with operators (\cite[p. 277, prop. 5.22]{Blue_book}). In particular, the complex algebra of a $\mathbf{S4}$-frame is an interior algebra. Since $\ModelPaioriginal$-models are based on $\mathbf{S4}$-frames, we can obtain in this way corresponding interior algebras. 

\begin{lemma}\label{lem.: paioriginal_main_correspondence}
    Let $\M=\pair{W,R,\pair{T_w,\oplus_w}_{w\in W},v,\{t_w\}_{w\in W}}$ be a rooted $\ModelPaioriginal$-model with $@$ its root, and let $\pair{\mathscr{P}(\mathbf{W}),\Box_R}$ be its frame's full complex algebra. Consider the $\ModelPaifull$-model $\pair{\pair{\mathscr{P}(\mathbf{W}),\Box_R},\pair{T_@,\oplus_@},N,v^E,t_@}$, with $v^E(p)=v(p)$ for all $p\in Var$. It holds that $\M,@\Vdash\phi\Leftrightarrow @\in v^E(\phi)$.
\end{lemma}
\begin{proof}
    We proceed by induction over the complexity of formulae. The base case is granted by the definition of $v^E$, while the Boolean cases are easy. When it comes to box, we are going to use the following property: $R[@]\subseteq v^E(\alpha)\Leftrightarrow\forall s\in R[@],\M,s\Vdash\alpha$. Again, we perform an induction over $\alpha$: for $p\in Var$, $\forall s\in R[@],\M,s\Vdash p$ iff, by definition of $v^E$, $\forall s\in R[@], s\in v^E(p)$ iff $R[@]\subseteq v^E(p)$. 
    
    The only non-trivial cases are those for box and arrow. Let $\alpha:=\Box\beta$ and observe that $\forall s\in R[@],\M,s\Vdash\Box\alpha$ iff $\forall s\in R[@],\M,s\Vdash\alpha$: the only if direction holds in virtue of (T), while for the if direction $\forall s\in R[@],\M,s\Vdash\alpha$ amounts to $\M,@\Vdash\Box\alpha$, then by (4) $\M,@\Vdash\Box\Box\alpha$, which implies $\forall s\in R[@],\M,s\Vdash\Box\alpha$. Now $\forall s\in R[@],\M,s\Vdash\alpha$ iff, by induction hypothesis, $R[@]\subseteq v^E(\alpha)$ iff $@\in \Box_R v^E(\phi)=\Box_R\Box_R v^E(\phi)=\Box_R v^E(\Box\phi)$, where the former identity holds by (EqT) and (Eq4), iff $R[@]\subseteq v^E(\Box\phi)$.

    Let $\alpha:=\beta\to\gamma$. $\forall s\in R[@],\M,s\Vdash\beta\to\gamma$ iff $\forall s\in R[@],\M,s\Vdash\Box(\beta\supset\gamma)$ and $t_s(\gamma)\leq_{\oplus_s}t_s(\beta)$. By the same reasoning above, $\forall s\in R[@],\M,s\Vdash\Box(\beta\supset\gamma)$ iff $\forall s\in R[@],\M,s\Vdash\beta\supset\gamma$, moreover by reflexivity $@\in R[@]$, so in particular $t_@(\gamma)\leq_{\oplus_@}t_@(\beta)$, and also the other direction holds, since by definition \ref{def.: Fine_PAI_model} $t_@(\gamma)\leq_{\oplus_@}t_@(\beta)$ and $s\in R[@]$ imply $t_s(\gamma)\leq_{\oplus_s}t_s(\beta)$. $\forall s\in R[@],\M,s\Vdash\beta\supset\gamma$ iff, by induction hypothesis, $R[@]\subseteq v^E(\beta\supset\gamma)$ iff $@\in \Box_R v^E(\beta\supset\gamma)=\Box_R v^E(\Box(\beta\supset\gamma))$ iff $R[@]\subseteq v^E(\Box(\beta\supset\gamma))=v^E(\beta\to\gamma)$, where the latter identity holds in virtue of $t_@(\gamma)\leq_{\oplus_@}t_@(\beta)$.

    Returning to the main claim, assume $\phi:=\Box\alpha$. $\M,@\Vdash\Box\alpha$ iff $\forall s\in R[@],\M,s\Vdash\alpha$ iff, by the property just proved, $R[@]\subseteq v^E(\alpha)$ iff $@\in \Box_R v^E(\phi)=v^E(\Box\phi)$.

    For the last case, let $\phi:=\alpha\to\beta$. $\M,@\Vdash\alpha\to\beta$ iff $\M,@\Vdash\Box(\alpha\supset\beta)$ and $t_@(\beta)\leq_{\oplus_@}t_@(\alpha)$ iff, by the previous cases, $@\in v^E(\Box(\alpha\supset\beta))$ and $t_@(\beta)\leq_{\oplus_@}t_@(\alpha)$ iff $@\in v^E(\alpha\to\beta)$.
\end{proof}

\begin{theorem}\label{th.: paioriginal_order_completeness}
    $\Paioriginal$ is the logic of order of the class of $\ModelPaifull$-models.
\end{theorem}
\begin{proof}
    The claim amounts to $\Gamma\vdash_\Paioriginal\phi\Leftrightarrow\Gamma\vDash_{\ModelPaifull}^\leq\phi$. Soundness is easy to check. Adequacy will be proved by contraposition, by proving that every Fine's $\ModelPaioriginal$-countermodel generate a respective $\ModelPaifull$-countermodel. Let $\Gamma\nvdash_\Paioriginal\phi$. By Fine's completeness of $\Paioriginal$ w.r.t. to its Kripke semantics (\cite{Fine86_analytic_implication}), there is a $\ModelPaioriginal$-countermodel $\M=\pair{W,R,\pair{T_w,\oplus_w}_{w\in W},v,\{t_w\}_{w\in W}}$ s.t. for some $@\in W, @\Vdash \gamma$ for every $\gamma\in\Gamma$, but $@\nVdash\phi$. Consider the submodel generated by the single point $@$, that is the rooted submodel of $\M$ with set of possible worlds $R[@]$. Let us build its corresponding $\ModelPaifull$-model $\pair{\pair{\mathscr{P}(\mathbf{R}[@]),\Box_{R\cap R[@]^2}},\pair{T_@,\oplus_@},N,v^E,t_@}$ described in lemma \ref{lem.: paioriginal_main_correspondence}. By the lemma, we know that $@\in v^E(\gamma)$ for all $\gamma\in\Gamma$, but $@\notin v^E(\phi)$. Therefore there is a point $x$, namely $\{ @ \}$, s.t. $x\leq v^E(\gamma)$ for each $\gamma\in\Gamma$ and $x\nleq v^E(\phi)$, that is $\Gamma\nvDash_{\ModelPaifull}^\leq\phi$.
\end{proof}

Finally, we prove that by weakening (Nec$_g$) for (Nec) in $\Paifull$ we obtain an equivalent complete calculus for Parry's logic. Let $\LocalPaifull=\pair{\lang_\mathbf{ML}^\to,\vdash_{\LocalPaifull}}$ be said logic.

\begin{theorem}
    $\LocalPaifull$ and $\Paioriginal$ are the same logic\footnote{See footnote \footref{footnote: Fine_language} for the considerations about the languages of the two different logics.}.
\end{theorem}
\begin{proof}
    That $\LocalPaifull\subseteq\Paioriginal$ follows from checking the soundness of $\LocalPaifull$ w.r.t. $\ModelPaioriginal$-frames and Fine's completeness theorem. We prove the other direction syntactically (we omit the subscript $\LocalPaifull$ under the turnstile for the remaining of the proof). $\LocalPaifull$ clearly derives all of $\Sfourlocal$ axioms and rules. Axioms (2) and (6) are shared between the logics. (5) is an application of (O5).
     
    For axiom (3), first observe that the deduction theorem for standard modal logic holds for $\Sfourlocal$ as well, since its rules are exactly the same of \textbf{K} and it contains classical theorems. By (C2) $(\phi\prec\chi)\land(\psi\prec\chi)\vdash(\phi\lor\psi)\prec\chi$; by (O4) $\vdash\neg\phi\prec\phi$ and by (C1) $\vdash\phi\prec(\phi\lor\psi)$, so using (O2) $\vdash\neg\phi\prec(\phi\lor\psi)$. Similarly we obtain $\vdash\neg\psi\prec(\phi\lor\psi)$, so by (C2) $\vdash(\neg\phi\lor\neg\psi)\prec(\phi\lor\psi)$. Again, by (O4) $\vdash\neg(\neg\phi\lor\neg\psi)\prec(\neg\phi\lor\neg\psi)$ and (C2) gives $\vdash\neg(\neg\phi\lor\neg\psi)\prec(\phi\lor\psi)$. Using this and (O2) we obtain $(\phi\prec\chi)\land(\psi\prec\chi)\vdash\neg(\neg\phi\lor\neg\psi)\prec\chi$, and by definition $\neg(\neg\phi\lor\neg\psi)=:\phi\land\psi$. By deduction theorem we conclude $\vdash((\phi\prec\chi)\land(\psi\prec\chi))\supset(\phi\land\psi)\prec\chi$.

    Axiom (4) is proved in three steps. Let $Var(\phi)\subseteq Var(\psi)$. Take any $p\in Var(\phi)$. By (O1) $\vdash p\prec p$. Let us denote by $Sub(\chi)$ the set of subformulae of $\chi$. For the second step, we prove by induction on $\beta\in Sub(\psi)$ s.t. $p\in Var(\beta)$ that $p\prec\beta$. If $\beta$ is a variable, then $\beta=p$ and the base case is done. If $\beta:=\neg\gamma$, then by induction hypothesis $\vdash p\prec\gamma$, so by (O4) $\vdash p\prec\neg\gamma$. The case is similar for $\beta:=\Box\gamma$, using (O5) instead. If $\beta:=\gamma_1\lor\gamma_2$, by induction hypothesis $\vdash p\prec\gamma_1$, hence by (C1) $\vdash p\prec\gamma_1\lor\gamma_2$. For $\beta:=\gamma_1\to\gamma_2$ we make the same reasoning, using (A5) to prove from $\vdash p\prec\gamma_1\lor\gamma_2$ that $\vdash p\prec\gamma_1\to\gamma_2$. This concludes the induction. Since, in particular, $\psi\in Sub(\psi)$, then $p\prec \psi$. For the last step, we prove that for all $\alpha\in Sub(\phi), \vdash\alpha\prec\psi$. We induce on the complexity of $\alpha$. The base case is provided by $\vdash p\prec\psi$. The cases for negation and box are similar to those in the previous induction. If $\alpha:=\gamma_1\lor\gamma_2$, by induction hypothesis $\vdash\gamma_1\prec\psi$ and $\vdash\gamma_2\prec\psi$, and using (C2) $\vdash(\gamma_1\lor\gamma_2)\prec\psi$. The case $\alpha:=\gamma_1\to\gamma_2$ uses the same reasoning, concluding by (A5) that $\vdash(\gamma_1\to\gamma_2)\prec\psi$. This proves $\phi\prec\psi$, therefore axiom (4) is derivable\footnote{That $\LocalPaifull$ coincides with original Parry's logic is obtained using \cite{Fine86_analytic_implication}'s proof of the equivalence between his formulation of $\Paioriginal$ and Parry's one.}.
\end{proof}

\n The reason why we mentioned $\LocalPaifull$ is that it is a proper finite axiomatic presentation of Parry's logic, devoid of axioms like (4) which is, strictly speaking, a schema of schemata.

As a concluding remark, we sketch a proof of the claim that $\FergusonPaiBox$ is the logic of order of $\ModelPaizero$-models. Recall that $\FergusonPaiBox$ is the expansion of Ferguson's logic $\FergusonPai$ with $\Box$ and the addition of axioms and the rule for (local) $\Sfourlocal$. We can replicate the proof that led to theorem \ref{th.: paioriginal_order_completeness}. Various minor details have to be fixed, starting from the definition of a $\ModelFergusonPai$-model we define a $\ModelFergusonPaiBox$-model for $\lang_\mathbf{ML}^\to$ as in definition \ref{def.: ferguson_ca_pai_model} adding the standard clause for box. The notion of generated submodel is immediately adapted from \ref{def.: generated_submodel}. Then we repeat the proofs of lemmas \ref{lem.: generated_submodel_satisfiability}-\ref{lem: Fine_surjective_models}-\ref{lem.: paioriginal_main_correspondence}; notice that Ferguson's models contain the family of homomorphisms $\{h_{ww'}:T_w\to T_{w'}\}_{w,w,'\in W}$, whose only role is to guarantee the persistence of topic inclusion throughout the successors of each point, therefore that property is not lost in the current iteration of the proofs. In lemma \ref{lem.: paioriginal_main_correspondence} the only difference in the corresponding $\ModelPaizero$-model is that the algebra of contents will be $\pair{T_@,\oplus_@,\gru_@}$. Completeness follows as in theorem \ref{th.: paioriginal_order_completeness}.

\begin{proposition}\label{th.: paioriginal_order_completeness}
    $\FergusonPaiBox$ is the logic of order of the class of $\ModelPaizero$-models.
\end{proposition}

\end{section}

\begin{section}{Logics of demodalized analytic implication}\label{sec: demodalized_logic}

In this section we will recover set-assignment semantics, showing how Epstein's dependence logics can be provided with an equivalent model theory within the new framework of generalized Epstein semantics. We start by focusing on Epstein's logic $\EpsteinD$, which is a reformulation of Dunn's logic of demodalyzed analytic implication $\DunnDai$ (\cite{Dunn72_demodalized_implication}). $\DunnDai$ is obtained by demodalizing $\Paioriginal$ via the axiom $\phi\to(\neg\phi\to\phi)$, which collapses the underlying modal logic into classical one, turning formulae containing boxes into graphical variants of the propositional ones. A good candidate for recovering this logic within generalized Epstein semantics is to demodalize $\Paifull$, using the axiom:
\begin{itemize}
    \item[(A6)] $\phi\supset\Box\phi$
\end{itemize}

\n We obtain $\Daifull^\Box=\pair{\lang_\mathbf{ML}^\to,\vdash_{\Daifull^\Box}}=\Paifull$+(A6). Using the definition of $\to$ provided by (A4), (T) and box transparency given by (O5), we derive $\vdash_{\Daifull^\Box}(\phi\to\Box\phi)\land(\Box\phi\to\phi)$. In this system any formula $\phi$ is logically equivalent to $\Box^n\phi,\forall n\in\mathbb{N}$ (i.e. $\phi$ prefixed by $n$ boxes), or, even more, $\phi$ is equivalent to $\phi^\Box$, where $\phi^\Box$ differs from $\phi$ at most for having some subformula $\Box^n\psi$ substituted for $\psi$ belonging to $\phi$. Following this path, it is straightforward to prove completeness of $\Daifull^\Box$ w.r.t. the class of $\ModelPaifull$-models in which the algebra of truth-values is an interior algebra where $\Box$ is the identity function. In this sense, the semantics for this logic is composed by structures which contain no additional information compared to standard Boolean algebras. Therefore we are going to follow a different approach, simplifying the language.

Let us return to the language $\lang_\mathbf{CL}^\to=\pair{\neg,\lor,\to}$ of classical logic extended with an intensional arrow operator. The logic $\Daifull=\pair{\lang_\mathbf{CL}^\to,\vdash_{\Daifull}}$ is given by the following Hilbert-style calculus:
\begin{itemize}
    \item[] \textbf{Axioms}
    \item[(A1)] $\phi\supset(\psi\supset\phi)$
    \item[(A2)] $(\phi\supset(\psi\supset\chi))\supset((\phi\supset\psi)\supset(\phi\supset\chi))$
    \item[(A3)] $(\neg\phi\supset\neg\psi)\supset(\psi\supset\phi)$
    \item[(A4$^D$)] $(\phi\to\psi)\equiv((\phi\supset\psi)\land(\psi\prec\phi))$
    \item[(A5)] $((\phi\to\psi)\prec(\phi\lor\psi))\land((\phi\lor\psi)\prec(\phi\to\psi))$
    \item[(O1)] $\phi\prec\phi$
    \item[(O2)] $((\phi\prec\psi)\land(\psi\prec\chi))\supset(\phi\prec\chi)$
    \item[(O3)] $(\phi\prec\psi)\supset(\phi\lor\psi\prec\psi)$
    \item[(O4)] $((\phi\prec\psi)\equiv(\phi\prec\neg\psi))\land((\phi\prec\psi)\equiv(\neg\phi\prec\psi))$
    \item[(C1)] $(\phi\prec\phi\lor\psi)\land(\psi\prec\phi\lor\psi)$
    \item[(C2)] $((\phi\prec\chi)\land(\psi\prec\chi))\supset((\phi\lor\psi)\prec\chi)$
    \item[(C3)] $((\phi\prec\chi)\land(\chi\prec\phi)\land(\psi\prec\zeta)\land(\zeta\prec\psi))\supset((\phi\to\psi)\prec(\chi\to\zeta))$
    \item[] \textbf{Rule}
    \item[(MP)] $\phi,\phi\supset\psi\vdash\psi$
\end{itemize}

\n $\Daifull$ is obtained by removing from $\Daifull^\Box$ the modal axioms, box transparency and changing (A4$^D$) for (A4), which now defines an analytic implication as material (and no long strict) implication plus content inclusion. Of course, necessitation is dismissed as well.

Since this is a demodalization, it will be sufficient to employ Boolean algebras in the semantics.
\begin{definition}\label{def.: dai-model}
    A $\ModelDaifull$-model is a $gE$-model $\pair{\B,\Se, N,v,s}$ for $\lang_\mathbf{CL}^\to$, which differs from a $\ModelPaifull$-model only for the facts that $\B$ is a Boolean algebra and:
    $$v(\phi\to\psi)=\begin{cases} 
	v(\neg\phi\lor\psi) & \text{if } s( N(\phi))\leq_\oplus s( N(\psi));\\
        0^{\B} & \text{otherwise.}
        \end{cases}$$
    
    \n The remaining parts are adapted from definition \ref{def.: pai-model} for the language $\lang_\mathbf{CL}^\to$.
\end{definition}

The completeness proof w.r.t. to this class of models is an adaptation of theorem $\ref{th.: paifull_completeness}$ for $\Paifull$. Throughout this section, we are going to denote $Fm_{\pair{\neg,\lor,\to}}$ by $Fm$, and we abbreviate $Fm_{\pair{\neg,\lor}}$ and $Fm_c$. For this completeness proof, the difference is only that we quotient the formula algebra for the language $\lang_\mathbf{CL}^\to$. For a prime $\Daifull$-theory $\Gamma$, we use the same congruence $\Omega\Gamma$ to obtain the Lindenbaum-Tarski algebra $\Fm_c\backslash\Omega\Gamma$ as the algebra of truth-values for the canonical model, which is now a Boolean algebra. Similarly we quotient the translation of $Fm_c$ through $N$ by $\sim_\Gamma$ and obtain the semilattice of contents $\langle N[Fm]/$$\sim_\Gamma,\oplus\rangle$. Lemma \ref{lemma: S4_lind_tarski_process} and corollary $\ref{cor.: paifull_order_correspondence}$ can be restated for $\Daifull$. 

\begin{definition}\label{def.: dai_canonical_model}
    Let $\Gamma$ be a prime $\Daifull$-theory. The canonical model for $\Gamma$ is the $\ModelDaifull$-model $\M^\Gamma=\langle\Fm_c/\Omega v_1[\Gamma],\langle N[Fm]/$$\sim_\Gamma,\oplus\rangle, N,v^\Gamma,\pi\rangle$, which is adapted from definition $\ref{def.: paifull_canonical_model}$, changing in $v^\Gamma=v_2\circ v_1$ the component $v_1:Fm\to Fm_c$ only for the case:
        $$v_1(\phi\to\psi)=\begin{cases} 
	v_1(\neg\phi\lor\psi) & \text{if } [ N(\psi)]_{\sim_\Gamma}\leq_\oplus[ N(\phi)]_{\sim_\Gamma};\\
         v_1(\neg\phi\lor\psi)\land\neg v_1(\neg\phi\lor\psi) & \text{otherwise.}
        \end{cases}$$
\end{definition}

The following is proved by a straightforward adaptation of the proof of theorem \ref{th.: paifull_completeness}:

\begin{theorem}\label{th.: daifull_completeness}
    $\Daifull$ is complete w.r.t. the class of $\ModelDaifull$-models.
\end{theorem}

Having completeness, it is easy to establish the coincidence between $\Daifull$ and Epstein's $\EpsteinD$, by proving, like it was done for theorem \ref{th.: uguaglianza_pai_e_ferguson}, that the logics are sound w.r.t. each other's complete semantics. 
\begin{theorem}\label{th.: uguaglianza_dai_e_epstein}
    $\Daifull$ and $\EpsteinD$ are the same logic.
\end{theorem}

\n Therefore $\ModelDaifull$-models are an alternative equivalent semantics for Epstein's (and Dunn's) logic. Returning to the original motivation for a generalized Epstein semantics, the above result shows how the only key ingredients for dependence logic $\EpsteinD$ are only the Boolean structure for truth-values and the join-semilattice for topics.

By switching the direction of content inclusion in the first clause for $\to$ of definition \ref{def.: dai-model}, we obtain models for Epstein's dual dependence logic \textbf{D} and the logic of equality of content \textbf{Eq}. $gdD$-models are those in which $v(\phi\to\psi)=v(\neg\phi\lor\psi)$ when $s( N(\phi))\geq_\oplus s( N(\psi))$, while in $gEq$-models the requirement becomes $s( N(\phi))\geq_\oplus s( N(\psi))$. These classes of models characterize, respectively, the logics $g\mathbf{dD}$ and $g\mathbf{Eq}$, which are obtained by replacing (A4$^D$) in $\Daifull$ by, respectively, (A4$^{dD}$) $(\phi\to\psi)\equiv((\phi\supset\psi)\land(\phi\prec\psi))$ and (A4$^{Eq}$) $(\phi\to\psi)\equiv((\phi\supset\psi)\land(\phi\prec\psi)\land(\psi\prec\phi))$. Completeness and equivalence of logics can be proved similarly to what has been done above. The following theorem states that the entirety of dependence logics can be recaptured within generalized Epstein semantics.
\begin{theorem}\label{th.: epstein_logics_correspondence}
    Let $\Le\in\{\EpsteinD,\mathbf{dD},\mathbf{Eq}\}$. $g\Le$ is complete w.r.t. the class of $gL$-models. $\Le$ and $g\Le$ are the same logic.
\end{theorem}

As a last consideration, we briefly see how to Ferguson's demodalized conditional-agnostic logic $\FergusonDai$ is recaptured within this semantics. The procedure is a straightforward adaptation, in reverse order, of the one that brought us from $\Paizero$ to $\Paifull$. Since we want the content of the arrow to be as general as possible, we drop axiom (A5), that states its identification with join. The resulting logic is denoted $\Daizero=\pair{\lang_\mathbf{CL}^\to,\vdash_{\Daizero}}$.
\begin{definition}\label{def.: dai0-model}
    A $\ModelDaizero$-model is a $gE$-model $\pair{\B,\Se, N,v,s}$ for $\lang_\mathbf{CL}^\to$, which differs from a $\ModelPaizero$-model only for the facts that $\B$ is a Boolean algebra and:
    $$v(\phi\to\psi)=\begin{cases} 
	v(\neg\phi\lor\psi) & \text{if } s( N(\phi))\leq_\oplus s( N(\psi));\\
        0^{\B} & \text{otherwise.}
        \end{cases}$$
    
    \n The remaining parts are adapted from definition \ref{def.: pai0-model} for the language $\lang_\mathbf{CL}^\to$.
\end{definition}

\n Notice that now $\Se$ is the expanded semilattice containing the operation $\gru$ that composes the topic of formulae containing arrow. The completeness proof is an easy reiteration of the one that brought to theorem \ref{th.: pai0_completeness}.
\begin{theorem}\label{th.: dai0_completeness}
    $\Daizero$ is complete w.r.t. the class of $\ModelDaizero$-models.
\end{theorem}

Again, by checking soundness to each other's complete semantics. we obtain that ours is a reformulation of Ferguson's logic.
\begin{theorem}
    $\Daizero$ and $\FergusonDai$ are the same logic.
\end{theorem}

\end{section}

\begin{section}{Conclusions}\label{sec: conclusion}

The following Hasse diagram describes the hierarchy of the studied logics: 

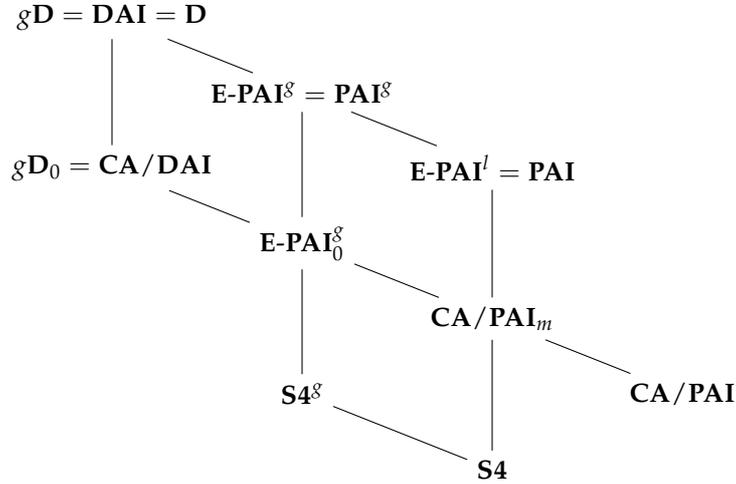
\begin{figure}[h]
\centering
\begin{tikzpicture}[scale=1]
\centering
  \node (one) at (-4.5,0) {$\Daifull=\DunnDai=\EpsteinD$};
  \node (one1) at (-4.5,-2) {$\Daifull_0=\FergusonDai$};
  \node (a) at (-2,-1) {$\Paifull=\FineGlobalPai$};
  \node (a1) at (0.5,-2) {$\LocalPaifull=\Paioriginal$};
  \node (b) at (-2,-3) {$\Paizero$};
  \node (b1) at (0.5,-4) {$\FergusonPaiBox$};
  \node (b2) at (3,-5) {$\FergusonPai$};
  \node (zero) at (-2,-5) {$\Sfourglobal$};
  \node (zero1) at (0.5,-6) {$\mathbf{S4}$};
  
  \draw (one) -- (a) -- (b) -- (one1) -- (one);
  \draw (a) -- (a1) -- (b1) -- (b) -- (zero) -- (zero1);
  \draw (zero1) -- (b1) -- (b2);
\end{tikzpicture}
\captionsetup{labelformat=empty}
\caption{Hasse diagram of the studied logics} \label{fig:M1}
\end{figure}

\n The links mean not only extension but also expansion. The position of the demodalized logics $\Daifull$ and $\Daizero$ is, strictly speaking, incorrect, since they are formulated in the $\Box$-free fragment of the language, yet as we mentioned both can be equivalently rewritten for $\lang_\mathbf{ML}^\to$ starting from the respective global Parry's systems immediately below them in the diagram and trivializing the $\Box$ via (A6).

The intent of this paper was two-fold: first, to answer - hopefully in a satisfactory way - to the open problem of finding a set-assignment semantics for Parry's logic of analytic implication; second, to show, via the case study of Parry's logic and nearby systems, how the proposed generalized Epstein semantics can be a helpful tool in logical investigations pertaining intensional phenomena. Ultimately, the main advantage of this framework resides in its modularity.

The semantics operates a distinction in treatment between the extensional and intensional fragments of the language. By individuating the latter fragment, we obtain a finer degree of control over the intensional behaviour of some connectives by splitting the universes of values assigned to formulae. Instead of a complex structure incorporating the entire semantics, we now have a separated algebra for intensional values simplifying the picture, as the value of an intensional formula is now determined with the contribution coming from the auxiliary algebra of contents. In order to investigate further logics, one of the two components of the semantics can be substituted without greatly affecting the other (up to adjusting the translation correlating the languages). E.g., we might be interested in the study of a Parry's system with an intuitionistic basis, which can be done by using Heyting algebras as the algebra of truth-values, or we might argue against the transparency of the content of some connective and decide to change the semilattice structure with something more articulated, like in Ferguson's conditional-agnostic logic. The framework is modular precisely in this sense. Hopefully this virtue will be helpful in future investigations.

\end{section}

\bibliographystyle{apalike}
\bibliography{Epstein}

\end{document}